\documentclass[12pt]{article}
\usepackage[british]{babel}
\usepackage{mathrsfs}
\usepackage{amsmath}
\usepackage{amssymb}
\usepackage{amsthm}
\usepackage{enumerate}
\usepackage{epic}
\usepackage{color}

\numberwithin{equation}{section}
\makeatletter
\renewcommand{\subsection}{\@startsection
{subsection}{2}{0mm}{\baselineskip}{-0.25cm}
{\normalfont\normalsize\bf}}
\makeatother

\newtheorem{theorem}{Theorem}[section]
\newtheorem{proposition}[theorem]{Proposition}
\newtheorem{lemma}[theorem]{Lemma}

\newtheorem{remark}[theorem]{Remark}
\newtheorem{result}[theorem]{Result}
\newtheorem{example}[theorem]{Example}

\def\PG{{\rm{PG}}}
\newcommand{\PSL}{\mbox{\rm PSL}}
\newcommand{\PGL}{\mbox{\rm PGL}}

\newcommand{\PGU}{\mbox{\rm PGU}}

\def\F{\mathbf F}

\def\cC{\mathcal C}
\def\cD{\mathcal D}
\def\cE{\mathcal E}
\def\cF{\mathcal F}
\def\cG{\mathcal G}
\def\cH{\mathcal H}

\def\cM{\mathcal M}
\def\cN{\mathcal N}
\def\cO{\mathcal O}

\def\cX{\mathcal X}

\def\cU{\mathcal U}

\def\cQ{\mathcal Q}

\def\fq{\mathbb F_{q^2}}
\def\gg{\mathfrak{g}}

\def\det{{\rm det}}

\newcommand{\qa}{\textstyle\frac{1}{4}}
\newcommand{\ha}{\textstyle\frac{1}{2}}
\newcommand{\aut}{\rm{Aut}}
\def\xfq2{{\cX(\mathbb{F}_{q^2})}}
\def\F+xfq2{{\cF^+(\mathbb{F}_{q^2})}}
\def\xpfq2{{\cX^+(\mathbb{F}_{q^2})}}
\def\xmfq2{{\cX^-(\mathbb{F}_{q^2})}}
\title{Hemisystems of the Hermitian Surface}
\author{G\'{a}bor Korchm\'{a}ros\footnote{G\'{a}bor Korchm\'{a}ros: gabor.korchmaros@unibas.it \hfill\newline\hspace*{1.4em} Dipartimento di Matematica, Informatica ed Economia -
Universit\`{a} degli Studi della Basilicata - Viale dell'Ateneo Lucano 10 - 85100 Potenza (Italy).}
\hfill\newline\hspace*{1.4em}
G\'{a}bor P. Nagy\footnote{G\'abor P. Nagy: nagyg@math.bme.hu  Department of Algebra, Budapest University of Technology and Economics Egry J\'ozsef utca 1, H-1111 Budapest, Hungary, and nagyg@math.u-szeged.hu  Bolyai Institute, University of Szeged (Hungary), Research supported by NKFIH-OTKA Grants 114614 and 119687.}
\and
Pietro Speziali:\footnote{Pietro Speziali: pietro.speziali@unibas.it \hfill\newline\hspace*{1.4em} Dipartimento di Matematica, Informatica ed Economia -
Universit\`{a} degli Studi della Basilicata - Viale dell'Ateneo Lucano 10 - 85100 Potenza (Italy).
}}
\date{}

\begin{document}
\date{}
\maketitle
\begin{abstract} We present a new method for the study of hemisystems of the Hermitian surface $\cU_3$ of $\PG(3,q^2)$. The basic idea is to represent generator-sets of $\cU_3$ by means of a maximal curve naturally embedded in $\cU_3$ so that a sufficient condition for the existence of hemisystems may follow from results about maximal curves and their automorphism groups. In this paper we obtain a hemisystem in $\PG(3,p^2)$ for each $p$ prime of the form
$p=1+16n^2$ with an integer $n$. Since the famous Landau's conjecture dating back to 1904 is still to be proved (or disproved), it is unknown whether there exists an infinite sequence of such primes. What is known so far is that just $18$ primes up to $51000$ with this property exist, namely $17,257,401,577, 1297,1601, 3137, 7057,13457,14401,15377,24337,25601,30977,$ $ 32401,33857,41617,50177.$
The scarcity of such primes seems to confirm that hemisystems of $\cU_3$ are rare objects.
\end{abstract}
\section{Introduction}
Construction of hemisystems on the Hermitian surface $\cU_3(q)$ of $\PG(3,q^2)$, as well as in any generalized quadrangle of order $(q^2,q)$, is a relevant issue because hemisystems give rise to important combinatorial objects which have been under investigation for many years, such as strongly regular graphs, partial quadrangles and 4-class imprimitive cometric
$Q$-antipodal association schemes that are not metric; for a thorough discussion about these connections see \cite{cp,bb,co}. Nevertheless, hemisystems of $\cU_3(q)$ appear to be rare. It had even been thought for a long time that just one hemisystem of $\cU_3$ existed, namely the (sporadic) example in $\PG(3,9)$ exhibited by B. Segre 1966 in his treatise \cite{segre};  see \cite{thas}. Actually, this was proved wrong in 2005 by Cossidente and Penttila \cite{cp} who were able to construct a hemisystem on $\cU_3(q)$ for every odd $q$. Their infinite family of hemisystems arises from commuting orthogonal and unitary polarities, first considered by Segre in \cite{segre}, and have the nice property of being left invariant by a subgroup of $\PGU(4,q)$ isomorphic to $\PSL(2,q^2)$.  Cossidente and Penttila \cite[Remark 4.4]{cp} also exhibited a hemisystem of $\cU_3(7)$ left invariant by a metacyclic group of order $516$, and a hemisystem of $\cU_3(9)$ left invariant by a metacyclic group of order $876$. Later on,  Bamberg, Giudici and Royle \cite{bb1} and \cite[Section 4.1]{bgr} found some more examples for smaller $q$'s, namely for $q=11,17,19,23,27$, left invariant by a cyclic group of order $q^2-q+1$. Very recently, Bamberg, Lee, Momihara and Xiang \cite{bb} came up with a new infinite family of hemisystems on $\cU_3(q)$ for every $q\equiv -1 \pmod 4$ left invariant by a metacyclic group $C_{(q^3+1)/4}:C_3$ that generalize the above examples. Their cyclotomic type of construction is carried out in the dual setting performed on the quadric $Q^-(5,q)$ and provides a (dual) hemisystem arising from a union of cyclotomic classes of $\mathbb{F}_{q^6}^*$. The proof of the existence of such hemisystems uses several tools from Finite geometry and Number theory, as well as, new results on (additive) character values of certain subsets of vectors obtained by computations involving Gauss sums. Finally, in \cite{cpa}, for every odd $q$, a hemisystem of $\cU_3$ left invariant by a subgroup of $\PGU(4,q)$ of order $(q+1)q^2$ is constructed.

 In this paper we adopt a different method for the construction of a hemisystem on $\cU_3(q)$. Our approach depends on the Natural embedding theorem \cite{KT} which states that any absolutely irreducible curve $\cX$ of degree $q+1$ lying on $\cU_3(q)$ is a (non-singular, $\mathbb{F}_{q^2}$-rational) maximal curve. In Section \ref{maximal} we show how such a maximal curve $\cX$ gives rise to a hemisystem of $\cU_3(q)$ when certain hypotheses are satisfied, and we also show that these hypotheses  may be stated in terms of the action of the automorphism  group of $\cX$ on the generators of $\cU_3(q)$.
For instance, these hypotheses are fulfilled  when $\cX$ is a rational curve, and in this case the arising hemisystem is projectively equivalent to the Cossidente-Pentilla hemisystem.

To obtain new hemisystems with this approach, a good choice for $\cX$ appears to be the Fuhrmann-Torres curve. We work out this case for $q\equiv 1 \pmod 8$ and we indeed find new hemisystems. Unfortunately, these hemisystems exist for some but not all $q$'s. We explain briefly why this occurs. Although our technical lemmas hold for any $q\equiv 1 \pmod 8$ (even for any $q\equiv 1 \pmod 4$), the last step of our procedure requires that the elliptic curve $\cE_3$ with affine equation $Y^2=X^3-X$ has as many as $(q-1)$ (or $(q+3)$) points over $\mathbb{F}_q$. If this occurs then our procedure provides a hemisystem.  On the other hand, for $q=p$ and $p\equiv 1 \pmod 8$, the latter case, namely $N_p(\cE_3)=p+3$, never occurs whereas $N_p(\cE_3)=p-1$ if and only if $p=1+16n^2$ with an integer $n$, and such primes exist. Whether an infinite sequence of such primes exists is still unknown, this question being related with Landau's conjecture as we have pointed out in Abstract; see also Section \ref{ser}. For $q=p^h$ with $h>1$ and $p\equiv 1 \pmod 4$, the problem whether $N_q(\cE_3)=\{q-1,q+3\}$ may occur for some $q$  has not been investigated so far in Number theory. Computer aided search suggests that the answer should be negative.

Our main result is stated in the following theorem.
\begin{theorem}
\label{teomain} Let $p$ be a prime number where $p=1+16a^2$ with an integer $a$. Then there exists a hemisystem in the Hermitian surface $\cU_3$ of $\PG(3,p^2)$ which is left invariant by a subgroup of $\PGU(4,p)$ isomorphic to $\PSL(2,p)\times C_{(p+1)/2}$.
\end{theorem}
The paper is organized in eight sections. The second is a background on hemisystems of Hermitian surfaces, the third contains the knowledge of maximal curves that are necessary for our purpose. In particular, the properties of maximal curves which are used in this paper are stated in Lemma \ref{lem18Ajun2017} in terms of Finite geometry. The fourth section describes our procedure based on a maximal curve $\cX$  and it gives necessary and sufficient conditions for the procedure to provide a hemisystem; see Conditions (A) and (B). It is also shown that they are fulfilled  when the automorphism group of $\cX$ satisfies Conditions (C), (D), and, for every point $P\in \cU_3(q)\setminus \cX$, Condition (E) (or (F)). Furthermore, if the maximal curve is rational then the procedure gives the Cossidente-Penttila hemisystem. In the remaining sections the case of the Fuhrmann-Torres maximal curve $\cX^+$ for $q\equiv 1 \pmod 4$ is worked out. By Theorem \ref{teo6lug2017}, Conditions (C) and (D) are fulfilled for $\cX^+$. Theorem \ref{em24lug2017} shows that Condition (E) is satisfied for many points of $\cU_3$ but there are exceptions. In Section 7, Condition (B) are investigated for exceptional points, under the extra condition that either $p\equiv 1 \pmod 8$ or $q$ is a square. Theorem \ref{teo5ago2017B} states that Condition (B) is satisfied
if and only if the elliptic curve of affine equation $Y^2=X^3-X$ has as many as $q-1$ (or $q+3$) points over $\mathbb{F}_q$. From this, Theorem \ref{teomain} follows.

\section{Background on Hermitian surfaces and hemisystems}
The Hermitian surface of $\PG(3,\fq)$ is defined to be the set of all self-conjugated points of the (non-degenerate) unitary polarity. Up to a change of homogeneous coordinates in $\PG(3,\fq)$,
the Hermitian surface consists of all points with homogeneous equation
$$X_0^{q+1}+X_1^{q+1}+X_2^{q+1}+X_3^{q+1}=0.$$
The number of points of $\PG(3,\fq)$ lying in $\cU_3$ is $(q^3+1)(q^2+1)$.
The linear collineation group preserving $\cU_3$ is $\PGU(4,q)$ and it acts on $\cU_3$ as a $2$-transitive permutation group.
A generator of $\cU_3$ is a line contained in $\cU_3$, and  the number of generators of $\cU_3$ is equal to $(q^3+1)(q+1)$. Through any point on $\cU_3$ there are exactly $q+1$ generators. A hemisystem of $\cU_3$ consists of
$\ha(q^3+1)(q+1)$ generators of $\cU_3$, exactly $\ha(q+1)$ of them through each point on $\cU_3$.

\section{Hermitian surfaces and maximal curves}
\label{maximal}
We recall some results on maximal curves that are used in the present paper. Proofs and more details are found in \cite[Chapter 10]{HKT}. For this purpose, the Hermitian surface $\cU_3$ has to be viewed as a smooth $2$-dimensional projective algebraic variety
in ${\mathrm{PG}}(3,\overline{\mathbb{F}}_{q^2})$ where
$\overline{\mathbb{F}}_{q^2}$ is the algebraic closure of
${\mathbb{F}}_{q^2}$. The points of $\cU_3$ with coordinates in ${\mathbb{F}}_{q^2}$ are the
{\emph{${\mathbb{F}}_{q^2}$-rational points of}} $\cU_3$. For
a point $A=(a_0,a_1,a_2,a_3)\in \PG(3,\overline{\mathbb{F}}_{q^2})$ lying on $\cU_3$, the tangent
(hyper)plane to $\cU_3$ at $A$ has equation
$$a_0^qX_0+a_1^qX_1+a_2^qX_2+a_3^qX_3=0.$$

Furthermore, the term of {\em algebraic curve defined over
${\mathbb{F}}_{q^2}$} means a projective, geometrically
irreducible, non-singular algebraic curve $\cX$ of
${\mathrm{PG}}(3,{q^2})$ viewed as a curve of
${\mathrm{PG}}(3,\overline{\mathbb{F}}_{q^2})$. Also,
$\cX({\mathbb{F}}_{q^{2i}})$ denotes the set of points of $\cX$
with all coordinates in ${\mathbb{F}}_{q^{2i}}$, called
${\mathbb{F}}_{q^{2i}}$-rational points of $\cX$. For a point
$P=(x_0,x_1,x_2,x_3)$ of $\cX$, the Frobenius image of $P$ is
defined to be the point $\Phi(P)=(x_0^{q^2},\ldots,x_n^{q^2})$ of $\cX$.
Then $P=\Phi(P)$ if and only if $P\in\cX({\mathbb{F}}_{q^2})$.

An algebraic curve $\cX$ defined over ${\mathbb{F}}_{q^2}$ is
 {\em ${\mathbb{F}}_{q^2}$-maximal} if the number $N_{q^2}$
of its ${\mathbb{F}}_{q^2}$-rational points attains the Hasse--Weil
upper bound, namely $N_{q^2}=q^2+1+2q\gg(\cX)$ where $\gg(\cX)$ denotes the
genus of $\cX$. In recent years, ${\mathbb{F}}_{q^2}$-maximal
curves have been the subject of numerous papers; a motivation for
their study comes from coding theory based on algebraic curves
having many points over a finite field. Here, only results on
maximal curves which are essential for the present investigation are
collected.

\begin{result} {\rm (Corollary to the Natural embedding theorem \cite{KT}; see also \cite{giuzzikorchmaros})}
\label{rs1} Any algebraic curve $\cX$ defined over ${\mathbb{F}}_{q^2}$ of
degree $q+1$ and contained in the non-degenerate Hermitian surface
$\cU_3$ is an ${\mathbb{F}}_{q^2}$-maximal curve. Furthermore,
\begin{itemize}
\item[\rm(i)] For any point $P\in\cX$,
the tangent hyperplane $\Pi_P$ to $\cU_3$ at $P$ coincides with the hyper-osculating plane to $\cX$ at
$P$, and
\begin{equation}
\label{hyperosc} \Pi_P\cap\cX= \left \{
\begin{array}{lll}
\{P\} & \mbox{\em for $P\in\cX({\mathbb{F}}_{q^2}) \/$},\\
\{P,\Phi(P)\} & \mbox{\em for $P\in\cX \backslash
\cX({\mathbb{F}}_{q^2})$.\/}
                         \end{array}
          \right.
\end{equation}
More precisely, for the intersection divisor $D$ cut out on $\cX$
by $\Pi_P$,
\begin{equation}
\label{hyperosc1} D= \left \{
\begin{array}{lll}
(q+1)P & \mbox{\em for $P\in\cX({\mathbb{F}}_{q^2}) \/$},\\
qP+\Phi(P) & \mbox{\em for $P\in\cX \backslash
\cX({\mathbb{F}}_{q^2})$.\/}
                         \end{array}
          \right.
\end{equation}
\item[\rm(ii)] The tangent line to $\cX$ at a point $P\in \cX({\mathbb{F}}_{q^2})$ is also a tangent line to $\cU_3$ at $P$, and it has no further common point with $\cU_3$.
\end{itemize}
\end{result}

The following example illustrates property (\ref{hyperosc}).
\begin{example}
\label{ex0} {\em{Up to a change of the projective frame in ${\mathrm{PG}}(3,{q^2})$,
the equation of $\cU_3$ may be written in the form $$X_1^{q+1}+X_2^{q+1}=X_0^qX_3+X_0X_3^q.$$ The ${\mathbb{F}}_{q^2}$-rational curve $\cX$ of degree $q+1$ consisting of all
points $$A(t)=\{(1,t,t^q,t^{q+1})\mid t\in\overline{\mathbb{F}}\}$$
together with the point $A({\infty})=(0,0,0,1)$ is defined over ${\mathbb{F}}_{q^2}$. Clearly, $\cX$ is contained in $\cU_3$ and hence it is a ${\mathbb{F}}_{q^2}$-maximal curve of genus $0$. Its tangent plane $\Pi_{A(t)}$ to
$\cU_3$ at $A(t)$ has equation
$$t^{q(q+1)}X_0+X_3-t^qX_1-t^{q^2}X_2=0.$$ To show that
(\ref{hyperosc}) holds for $A(t)$, it is necessary to check that
the equation $$t^{q(q+1)}+u^{q+1}=t^qu+t^{q^2}u^q$$ has only two
solutions in $u$, namely $u=t$ and $y=t^{q^2}$. Replacing $u$ by
$v+t$, the equation becomes $v^{q+1}+v^qt=t^{q^2}v^q$. For
$v\neq 0,$ that is, for $u\neq t$, this implies $v=t^{q^2}-t$,
proving the assertion. For $A(\infty)$, the tangent hyperplane
$\Pi_{A(\infty)}$ has equation $X_0=0$. Hence it does not meet
$\cX$ outside $A(\infty)$, showing that (\ref{hyperosc}) also
holds for $A(\infty)$.}}
\end{example}
\begin{remark}
\label{rem18ago2017}{\em{The classical terminology regarding rational cubic curves in the three-dimensional real space is also used for $\cX$ in the sequel.  A (real) chord of $\cX$ is a line in ${\mathrm{PG}}(3,{q^2})$ which meets
$\cX({\mathbb{F}}_{q^2})$ in at least two distinct points. An imaginary chord of $\cX$ is a line in ${\mathrm{PG}}(3,{q^2})$ which joins a point $P\in \cX({\mathbb{F}}_{q^4}) \setminus \cX({\mathbb{F}}_{q^2})$ to its
conjugate, that is, its Frobenius image $\Phi(P)\in \cX({\mathbb{F}}_{q^4}) \setminus \cX({\mathbb{F}}_{q^2})$.}}
\end{remark}
\begin{lemma}
\label{lem18Ajun2017} Let $\cX$ be an ${\mathbb{F}}_{q^2}$-maximal curve naturally embedded in the Hermitian surface $\cU_3$. Then
\begin{itemize}
\item[(i)] No two distinct points in $\cX(\mathbb{F}_{q^2})$ are conjugate under the unitary polarity associated with $\cU_3$.
\item[(ii)] Any imaginary chord of $\cX$ is a generator of $\cU_3$ which is disjoint from $\cX$.
\item[(iii)] For any point $P\in \cU_3$ in $\PG(3,\fq)$, let $\Pi_P$ be the tangent plane to $\cU_3$ at $P$. If $P\not\in\cX(\mathbb{F}_{q^2})$ then $\pi_P\cap \cX$ consists of $q+1$ pairwise distinct points which are in $\cX(\mathbb{F}_{q^4})$.
\end{itemize}
\end{lemma}
\begin{proof} From (\ref{hyperosc}), the tangent plane to $\cU_3$ at a point in $\cX(\mathbb{F}_{q^2})$ contains no more points of $\cX$. This yields (i). To show (ii) observe first that
the line through the points $P$ and $\Phi(P)$ is preserved by $\Phi$ and hence it is a line of $\PG(3,\fq)$. On the other hand, a line of $\PG(3,\fq)$ which is not a generator of $\cU_3$ meets $\cU_3$ in either $1$, or $q+1$ points in $\PG(3,\fq)$. In the former case, the line is a tangent to $\cU_3$ and cannot have any further common point with $\cU_3$ even in $\PG(3,\overline{\mathbb{F}})$. The last claim also holds in the latter case since the generators of $\cU_3$ are the only lines in $\PG(3,\overline{\mathbb{F}})$ containing more than $q+1$ points of $\cU_3$.
Therefore the first claim (ii) is proven. To show the second one, it is sufficient to observe that (i) of Result \ref{rs1} yields that a generator of $\cU_3$ with an $\mathbb{F}_{q^2}$-rational point of $\cX$ contains no further point from $\cX$.
To show (iii), the starting point is that the intersection of $\Pi_P\cap \cU_3$ splits into $q+1$ generators. If $g$ is one of them then we show that the common points of $g$ and $\cX$ are in $\PG(3,q^4)$. By way of contradiction, let $Q\in g\cap \cX$ such that $\Phi^2(P)\not\in \{P,\Phi(P)\}$. Then $P,\Phi(P)$ and $\Phi^2(P)$ are three pairwise distinct points on $g$. From (\ref{hyperosc}), $g$ is contained in the tangent plane $\Pi_P$ to $\cX$ at $P$ but $\Pi_P$ contains only two points of $\cX$, namely $P$ and $\Phi(P)$. This contradiction shows that $\Pi \cap \cX$ is contained in $\cX(\mathbb{F}_{q^4})$. To complete the proof of (iii) we need to show that for any point $Q\in \Pi \cap \cX$, the intersection number $I(Q,\cX\cap \Pi_P)$ is equal to $1$. If this number was at leats $2$ then the tangent line $\ell$ to $\cX$ at $Q$ would be contained in $\Pi_P$. If $Q\in\cX(\mathbb{F}_{q^2})$, a generator of $\cX$ through $P$ other than $g$ would meet $\ell$ in a point distinct from $Q$. But this contradicts (ii) in Result \ref{rs1}. If $Q\in\cX(\mathbb{F}_{q^4})\setminus\cX(\mathbb{F}_{q^2})$, $\Pi_P$ would contain the tangent to $\cX$ at $Q$ as well as $\Phi(Q)$. From (i) of Result \ref{rs1}, $\Pi_P$ would coincide with the tangent plane to $\cX$ at $Q$ whence $I(Q,\cX\cap \Pi_P)=q$. This holds true for $\Phi(Q)$ and hence $I(\Phi(Q),\cX\cap \Pi_P)=q$. But then $I(Q,\cX\cap \Pi_P)+I(\Phi(Q),\cX\cap \Pi_P)>q+1$
in contradiction with B\'ezout's theorem.
\end{proof}
\section{Hemisystems arising from maximal curves}
\label{hema}
In this section $\cX$ is an $\mathbb{F}_{q^2}$-maximal curve of genus $\gg(\cX)$ naturally embedded in the Hermitian surface $\cU_3$, and $\cH$ is the set of all imaginary chords of $\cX$. Furthermore, for any point $P\in \PG(3,\fq)$ lying in $\cU_3\setminus\cX(\mathbb{F}_{q^2})$, $n_P(\cX)$ denotes the number of generators of $\cU_3$ through $P$ which contain an $\mathbb{F}_{q^2}$-rational point of $\cX$. A set $\cM$ of generators of $\cU_3$ is a \emph{half-hemisystem on $\cX$} if the following properties hold:
{\emph{
\begin{itemize}
\label{18jun2017}
\item[\rm(A)] Each $\mathbb{F}_{q^2}$-rational point of $\cX$ is incident with exactly $\ha(q+1)$ generators in $\cM$.
\item[\rm(B)] For any point $P\in\cU_3\setminus\cX(\mathbb{F}_{q^2})$ lying in $\PG(3,\fq)$, $\cM$ has as many as $\ha n_P(\cX)$ generators through $P$ which contain an $\mathbb{F}_{q^2}$-rational point of $\cX$.
\end{itemize}
}}
By (i) of Lemma \ref{lem18jun2017}, each generator of  $\cU_3$ contains zero or one $\mathbb{F}_{q^2}$-rational points. Since $\cX$ has exactly $N_{q^2}=q^2+1+2q\gg(\cX)$ $\mathbb{F}_{q^2}$-rational points, a half-hemisystem $\cM$ consists of $\ha (q+1)N_{q^2}$ generators. Furthermore,  $|\cX(\mathbb{F}_{q^4})|=q^4+1-2q^2\gg(\cX)$ whence
$$|\cX(\mathbb{F}_{q^4})\setminus \cX(\mathbb{F}_{q^2})|=(q^2+q)(q^2-q-2\gg(\cX)).$$
Therefore, (ii) of Lemma \ref{lem18jun2017} shows that the set $\cH$ consists of $\ha((q^2+q)(q^2-q-2\gg(\cX)))$ generators of $\cU_3$ none of them lying in $\cM$. Therefore $\cM\cup \cH$ comprises as many as
$$\ha (q+1)(q^2+1+2q\gg(\cX))+\ha((q^2+q)(q^2-q-2\gg(\cX)))=\ha(q^3+1)(q+1)$$
generators of $\cU_3$.
\begin{proposition}
\label{lem18jun2017} $\cM\cup \cH$ is a hemisystem of $\cU_3$.
\end{proposition}
\begin{proof} By (A) it is enough to deal with a point $P\in \PG(3,\fq)$ lying on $\cU_3\setminus \cX(\mathbb{F}_{q^2})$. By B\'ezout's theorem, the tangent plane $\Pi_P$ to $\cU_3$ at $P$ cuts out on $\cX$ as many as $q+1$ points. Since $n_P(\cX)$ of these points are $\mathbb{F}_{q^2}$-rational points of $\cX$, (iii) of Lemma \ref{lem18jun2017} shows that the remaining $q+1-n_P(\cX)$ points determine
$\ha(q+1-n_P(\cX))$ pairwise distinct generators in $\cH$. Now from (B), the number of generators in $\cM\cup \cH$ containing $P$ is equal to $\ha n_P(\cX)+\ha (q+1-n_P(\cX))=\ha (q+1).$
\end{proof}
If Condition (A) is satisfied, on can ask for some sufficient conditions, for instance in terms of automorphisms of $\cX$, that ensure the existence of points $P$ satisfying Condition (B). If such hypotheses are fulfilled for many points $P$, it is conceivable that Proposition \ref{lem18jun2017} holds. To find out what hypothesis may be useful, some preliminaries are needed.

From \cite[Theorem 3.7]{KT}, the automorphism group $\aut(\cX)$ of $\cX$ may be viewed as the subgroup of $\PGU(4,q)$ which leaves $\cX(\mathbb{F}_{q^2})$ invariant.
Therefore, $\aut(\cX)$  also acts as a permutation group on $\cX(\mathbb{F}_{q^{2i}})$ for $i=1,2,\ldots$. Take a subgroup $\mathfrak{G}$ of $\aut(\cX)$.
 Let $o_1,\ldots,o_r$ denote the $\mathfrak{G}$-orbits on $\cX(\mathbb{F}_{q^{2}})$. Also, for $1\le j \le r$, let $\cG_j$ denote the set of all generators of $\cU_3$ meeting $o_j$. Then the generator-set $\cG=\cup\,\cG_j$
has size $(q+1)|\cX(\fq)|$ and it consists of all generators of $\cU_3$ which meet $\cX(\fq)$. Obviously, $\mathfrak{G}$ leaves each $\cG_j$ invariant.
\begin{proposition}
\label{122jun2017} With the above notation, assume that the subgroup $\mathfrak{G}$ of $\aut(\cX)$ fulfills the hypotheses {\emph{(C)}} and {\emph{(D)}}:
\begin{itemize}
\item[\rm(C)] $\mathfrak{G}$ has a subgroup $\mathfrak{H}$ of index $2$ such that $\mathfrak{G}$ and $\mathfrak{H}$ have the same orbits $o_1,\ldots,o_r$ on $\xfq2$.
\item[\rm(D)] For any $1\le j \le r$,  $\mathfrak{G}$ acts transitively on $G_j$ while $\mathfrak{H}$ has two orbits on $G_j$.
\end{itemize}
For a point $P\not\in \cX(\mathbb{F}_{q^2})$ lying on a generator in $G$, if
\begin{itemize}
\item[\rm(E)] some element in $\mathfrak{G}_P$ is not in $\mathfrak{H}_P$,
\end{itemize}
then $P$ satisfies Condition {\emph{(B)}}.
\end{proposition}
\begin{proof} We show first that Condition (A) is satisfied. Take a point $Q\in \xfq2$, and let $o_j$ be the $\mathfrak{G}$-orbit containing $Q$. Let $\ell$ be the tangent to $\cX$ at $Q$. From  (ii) of Result \ref{rs1}, $\ell$ is not a generator of $\cU_3$. Its polar $\ell^\perp$ is another line in the tangent plane $\Pi_Q$ at $Q$, and  both lines $\ell$ and $\ell^\perp$ are $\mathfrak{G}_Q$-invariant. In terms of  the cyclic linear group $\bar{\Sigma}$ induced by $\mathfrak{G}_Q$ on the pencil $\cQ$  with center $Q$ in $\Pi_Q$, the action of $\bar{\Sigma}$ has two fixed lines, $\ell$ and $\ell^\perp$, and, by the first condition in (D), it also has a regular orbit of size $q+1$ consisting of all generators through $Q$, that is, $\bar{\Sigma}$ is a regular permutation group on the generator-set $\cG_j$. In particular, $\bar{\Sigma}$ has order $q+1$ and it has a unique subgroup $\bar{\Psi}$ of index $2$. Therefore, $\cG_j=\cM_j\cup \cM_j'$ where $\cM_j$ and $\cM_j'$ are the $\bar{\Psi}$-orbits on $\cG_j$. On the other hand, from the second conditions in (C) and (D), $\bar{\Psi}$ is the permutation group induced by $\mathfrak{H}_Q$ on the pencil $\cQ$. It turns out that $\cM=\cup\,\cM_j$ and $\cM'=\cup\,\cM_j'$ are the $\mathfrak{H}$-orbits on $G$. Now take a point $P$ lying on a generator $g\in \cG$. Let $\cN$ denote the set of generators through $P$ which meet $\xfq2$. Assume that there exists $\gamma\in \mathfrak{G}\setminus \mathfrak{H}$ with $\gamma(P)=P$. Since $\cG$ splits into the $\mathfrak{H}$-orbits $\cM$ and $\cM'$, $\gamma$ interchanges $\cM$ and $\cM'$. Hence $\gamma$ interchanges the generator sets $\cN\cap \cM$ and $\cN\cap \cM'$. This yields that these sets have the same size whence the claim follows.
\end{proof}
\begin{remark}
\label{rem29jun}
{\em{We point out that Condition (E) is satisfied if and only if}}
\begin{itemize}
\item[\rm(F)] $\mathfrak{G}_P$ has an element that fixes no point in $\xfq2$.
\end{itemize}
{\em{Let $g\in \cG$ be a generator through $P$ meeting $\xfq2$ in the point $Q$.
From the proof of Proposition \ref{122jun2017}, the stabilizer of $g$ in $\mathfrak{H}_Q$ is also the stabilizer of $g$ in $\mathfrak{G}_Q$. Therefore, (F) and (E) are equivalent.}}
\end{remark}
\begin{remark}
\label{remA29jun} {\em{The subgroup of $\PGU(4,q)$ preserving $\cG$ may happen to be larger than $\aut(\cX)$.
If this is the case and $\Psi$ is such a subgroup containing $\mathfrak{H}$ as a normal subgroup,
Condition (D) may be replaced by Condition}}
\begin{itemize}
\item[\rm(G)] Some element in $\Psi_P\setminus \mathfrak{H}_P$ does not leave $\cM$ invariant.
\end{itemize}
{\em{In fact, let $\delta\in\Psi_P\setminus \mathfrak{H}_P$. Since $\gamma$ normalizes $\mathfrak{H}$, Condition (G) yields that $\delta$ interchanges $\cM$ and $\cM'$.
}}
\end{remark}

\begin{remark}
{\em {The above results apply to the curve $\cX$ in Example \ref{ex0} and provide a new proof for the Cossidente-Penttila hemisystem. Since $\gg(\cX)=0$, the automorphism group $\aut(\cX)$ of $\cX$ defined over $\mathbb{F}_{q^2}$ is isomorphic to $\PGL(2,q^2)$ and acts on $\mathbb{F}_{q^2}$ as $\PGL(2,q^2)$ in its natural $3$-transitive permutation representation. Its action on the remaining points of $\cU_3$ has two orbits, say  $\cO_1$ and $\cO_2$, of length $\ha(q+1)q^2(q^2+1)$ and $\ha(q-1)q^2(q^2+1)$ respectively. The unique subgroup of $\aut(\cX)$ of index $2$  is isomorphic to $\PSL(2,q^2)$ which has the same three orbits on $\cU_3$ as $\aut(\cX)$ does. Choose $\aut(\cX)$ for $\mathfrak{G}$ and its unique subgroup isomorphic to $\PSL(2,q^2)$  for $\mathfrak{H}$. Then Condition (C) is satisfied with $r=1$. To show that $\mathfrak{G}$ also satisfies the first conditions in (D), we need to prove that $\mathfrak{G}_Q$ for $Q\in \xfq2$ acts regularly on the set of all generators through $Q$. Since $\mathfrak{G}$ is (triply)-transitive on $\xfq2$, we may assume $Q=A(\infty)=(0,0,0,1)$. Then $\mathfrak{G}_Q$ consists of $q^2(q-1)^2$ linear collineations $\gamma_{\alpha,\lambda}$ with $\alpha\in {\mathbb{F}_{q^2}}, \lambda\in {\mathbb{F}_{q^2}^*}$ associated to the matrices
$$
\left(
  \begin{array}{cccc}
    1 & 0 & 0 & 0 \\
    \alpha & \lambda & 0 & 0 \\
    \alpha^q & 0 & \lambda^q & 0 \\
    \alpha^{q+1} & \lambda\alpha^q & \lambda^q\alpha & \alpha^{q+1} \\
  \end{array}
\right)
$$
Furthermore, the generators through $Q$ are the lines $\ell_j$ with
$$\ell_j=\{(0,1,\theta_j,\mu)|\mu\in \mathbb{F}_{q^2}\} \cup \{(0,0,0,1)\}$$
where $j=1,2,\ldots, q+1$ and $\theta_j^{q+1}=-1$. Now, a direct computation proves that $\mathfrak{G}_Q$ acts regularly on the set of all generators through $Q$. Also, since $\gamma_{\alpha,\lambda}\in \mathfrak{H}_Q$ if and only if $\lambda$ is a square in $\mathbb{F}_{q^2}^*$, a direct computation also shows that $\mathfrak{H}_Q$ has two orbits of generators through $Q$, both of length $\ha(q+1)$. Therefore, Conditions (C) and (D) are satisfied. Also, for a point $P\in g$ with a generator $g$ through $Q$, the $2$-point stabilizers $\mathfrak{G}_{Q,P}$ and $\mathfrak{H}_{Q,P}$ coincide, and $|\mathfrak{G}_{Q,P}|=q-1$.
From these computation we also have that $\mathfrak{G}$ acts transitively on the set $S$ of points covered by the generators meeting $\xfq2$. From this, either $S=\cO_1$ or $S=\cO_2$. Since the stabilizer of a point in $\cO_2$ has order $2(q+1)$, the latter case cannot actually occur. Therefore, $S=\cO_1$, whence $|\mathfrak{G}_P|=2(q-1)$, and $|\mathfrak{H}_P|=q-1$. Now, the hypotheses of Propositions \ref{lem18jun2017} and \ref{122jun2017} are fulfilled, and we obtain the Cossidente-Penttila hemisystem.
It may also be shown with a double counting argument that the number $n_P(\cX)$ generators through $P$ meeting $\xfq2$ is equal to $2$.}}
\end{remark}
\section{Our procedure on the Fuhrmann-Torres maximal curve for $q\equiv 1 \pmod 4$}
\label{secsome}
Our aim is to show that the method introduced in Section \ref{hema} also applies to another family of $\mathbb{F}_{q^2}$-rational curve. From now on we assume $q\equiv 1 \pmod 4$.

Up to a change of the projective frame in ${\mathrm{PG}}(3,{q^2})$,
the equation of $\cU_3$ may also be written in the form
\begin{equation}
\label{eq017jun2017}
X_1^{q+1}+2X_2^{q+1}=X_3^qX_0+X_3X_0^q.
\end{equation}
In $\PG(2,q^2)$ with homogenous coordinates $(x,y,z)$. let $\cF^+$ be the plane curve of genus $\qa(q-1)^2$ with equation
\begin{equation}
\label{eq17jun2017}
y^qz-yz^q=x^{(q+1)/2}.
\end{equation}
A non-singular model $\cX^+$ of $\cF^+$ is given by the morphism
$\varphi\,:\,\cF^+ \rightarrow {\rm PG}(3,\overline{\mathbb{F}}_{q^2})$ defined by the coordinate functions
$$X_0=z^2,\  X_1=xz,\  X_2=yz,\  X_3=y^2.$$
The curve $\cX^+$ lies in the
Hermitian surface $\cU_3$ and it is an $\mathbb{F}_{q^2}$-rational curve of degree $q+1$ naturally embedded in $\cU_3$.
Furthermore, $\xpfq2$ has size $\ha(q^3+q+2)$ and it contains the set $\Omega$ cut out on $\cX^+$ by the plane $\pi$ of equation $X_1=0$; equivalently $\Omega$ is the complete intersection in $\pi$ of the conic $\cC$ with equation $X_0X_3-X_2^2=0$ and the Hermitian curve $\cH(2,q^2)$ with equation $X_0^qX_3+X_0X_3^q-2X_2^{q+1}=0$. Therefore, $\xpfq2$ splits into $\Omega$ and the set $\Delta^+$ consisting all $\fq$-rational points of $\cX^+$ not lying on $\pi$ where $|\Omega|=q+1$ and $|\Delta^+|=\ha(q^3-q)$.
The point $X_\infty=(0,1,0,0)$ is not in $\cU_3$. Furthermore, a line through $X_\infty$ meets $\xpfq2$ in either $\frac{1}{2}(q+1)$ or $1$ or $0$ points. More precisely,  there are exactly $q^2-q$ lines
through $X_\infty$ sharing $\frac{1}{2}(q+1)$ points with $\xpfq2$, and $q+1$ lines having just one point in $\xfq2$. The former lines meet the plane $\pi$ in the points of the conic $\cC$ not lying on $\Omega$; the latter in the points of $\Omega.$

The curve $\cX^+$ also lies on the cone $\cQ_3$ with vertex $X_\infty$ which projects the conic $\cC$, and the complete intersection of $\cU_3$ and $\cQ_3$ splits into $\cX^+$ and another $\fq$-rational maximal curve $\cX^-$. Actually, $\cX^-$ is a non-singular model of the plane curve $\cF^-$ of
equation
\begin{equation}
\label{eq30jun2017}
y^qz-yz^q=-x^{(q+1)/2}.
\end{equation}
by the the morphism $\varphi$. The above properties of $\cX^+$ hold true for $\cX^-$ when $^+$ is replaced by $^-$. Also, $\Omega$ is the set of all common points of $\cX^+$ and $\cX^-$.
\subsection{Conditions (C) and (D) on $\Delta^+$.}
\label{sub1}
According to our notation introduced in Section \ref{hema}, let $\cG$ denote the set of all generators of $\cU_3$ which meet $\cX^+$.
\begin{lemma}
\label{lem30lug}
$\cG$ is also the set of all generators meeting $\cX^-$.
\end{lemma}
\begin{proof}
Fix any point $P_{u,v}=(1,u,v,v^2)$ on $\cX^+$. Then $u^{(q+1)/2}=v^q-v$. Obviously, $P\not \in \cX^-$ may be assumed. Then $v\in \fq\setminus \mathbb{F}_{q}.$  The tangent plane to $\cU_3$ at $P_{u,v}$ contains the point $Q_{s,t}=(1,s,t,t^2)$ of $\cX^-$ (with $-s^{(q+1)/2}=t^q-t$) if and only if
\begin{equation}
\label{eq130jun2017}
u^qs-(t-v^q)^2=0.
\end{equation}
By Result \ref{rs1}, it is enough to prove that Equation (\ref{eq130jun2017}) has exactly $q+1$ pairwise distinct solutions. Assume that $(s,t)$ is such a solution. Then $$(su^q)^{(q+1)/2}=
s^{(q+1)/2}{(u^q)}^{(q+1)/2}=-(t^q-t)(v^q-v)^q=(t^q-t)(v^q-v),$$ whence
\begin{equation}
\label{eq1lug2017} (t-v^q)^{q+1}=(t^q-t)(v^q-v);
 \end{equation}
equivalently $(t+v)^{q+1}=2(tv+(tv)^q$.
Therefore, $(v,t)$ is a root of the polynomial
\begin{equation}
\label{eqA1lug2017} F(V,T)=(T+V)^{q+1}-2(TV+(TV)^q).
\end{equation}
Let $\cF$ be the plane curve with affine equation $F(V,T)=0$. The singular points of $\cF$ are its two points at infinity, $V_\infty$ and $T_\infty$, each being the center of exactly one branch
defined over $\mathbb{F}_q$, together with its points on the line of equation $V=T$, each being the center of exactly two branches, and all such branches are defined over $\mathbb{F}_q$ so that
$\cF(\mathbb{F}_q)$ has size $2(q+1)$.

Let $\cH$ be the Hermitian curve with affine equation $X^{q+1}=Y^q+Y$. The morphism $\psi,:\,\cF \rightarrow \cH$ defined by $X=V+T,\, Y= 2VT$ is a Galois cover of degree $2$ associated to
the involutory automorphism $\alpha:\, (V,T)\rightarrow (T,V)$. Since $v\in \fq\setminus \mathbb{F}_{q}$, the point $P(v,t)$ of $\cF$ is  non-singular, and hence $\psi(P(v,t))=(v+t,2tv)\in \cH$. If $\tau$ runs over
$\fq\cup \{\infty\}$ then the points $R_\tau=(v+\tau,2v\tau)$ describe the line $\ell_v$ of equation $Y=2Xv-v^2$. A direct computation shows that $\ell_v$ is not a tangent to $\cH$. Therefore, $|\ell_v\cap \cH|=q+1$.
This shows that for any fixed $v\in \fq\setminus \mathbb{F}_{q}$ Equation (\ref{eq1lug2017}) has exactly $q+1$ pairwise distinct solutions in $t$. Therefore, $\cF(\fq)\setminus \cF(\mathbb{F}_q)$
has size $(q^2-q)(q+1)$.

It remains to show that for any fixed $v\in \fq\setminus \mathbb{F}_{q}$, there exists exactly $\ha(q+1)$ values of $u$ such that $v^q-v=u^{(q+1)/2}$. Equation $v^q-v=u^{(q+1)/2}$ has $\ha(q+1)$ solutions in $u\in \overline{\mathbb{F}}_{q^2}$, and we show that every such $u$ falls in $\fq$. From $v^q-v-u^{(q+1)/2}=0$,
$${(u^{(q+1)/2})}^{q}=v^{q^2}-v^q=-(v^q-v)=-u^{(q+1)/2}$$
whence $$u^{(q+1)/2}({(u^{(q+1)/2}})^{q-1}+1)=0$$
which yields $u^{q^2-1}=1$, that is, $u\in \fq$.
\end{proof}
\begin{remark}
\label{rem30lug}
{\em{The proof of Lemma \ref{lem30lug} also shows that the set $\cG$ of all generators meeting $\cX^+(\fq)$ splits into two subsets: $\cG=\cG_0\cup\cG_1$ where $\cG_0$ is the set of the $(q+1)^2$ generators meeting $\Omega$ while $\cG_1$ is the set of the $\ha(q^3-q)(q+1)$ generators meeting both $\Delta^+$ and $\Delta^-$. The same holds for $\cX^-(\fq)$.}}
\end{remark}
This gives a useful representation of $\cG_1$ stated in the following lemma.
\begin{lemma}
\label{lem1lug2017}  The generator-set $\cG_1$ consists of the lines $g_{u,v,s,t}$ spanned by the points $P_{u,v}\in \Delta^+$ and $Q_{s,t}\in \Delta^-$ with
\begin{equation}
\label{eqC24lug}
F(v,t)=(v+t)^{q+1}-2(vt+(vt)^q)=0,
\end{equation}
and
\begin{equation}
\label{eq25Alug}
u^{(q+1/2}=v^q-v,\quad -s^{(q+1)/2}=t^q-t, \quad s=\frac{(t-v^q)^2}{u^q}.
\end{equation}
\end{lemma}
According to our approach described in Section \ref{hema}, we need some group-theoretic properties of $\cF$.
\begin{lemma}
\label{lem1Alug1} $\aut(\cF)$ contains a subgroup $\Psi\cong \PGL(2,q)$ that acts faithfully on the set $\cF(\fq)\setminus \cF(\mathbb{F}_q)$ as a sharply transitive permutation group.
\end{lemma}
\begin{proof} A direct computation shows that the following maps are automorphisms of $\cF$.
\begin{itemize}
\item[(i)] $(V,T)\rightarrow (V+\alpha,T+\alpha),\, \alpha\in \mathbb{F}_q$.
\item[(ii)] $(V,T)\rightarrow (\mu V,\mu T),\, \mu\in \mathbb{F}_q^*$.
\item[(iii)] $(V,T) \rightarrow (V^{-1},T^{-1})$.
\end{itemize}
 These automorphisms are defined over $\mathbb{F}_q$ and generate a group $\Psi\cong \PGL(2,q)$ that preserves $\cF(\fq)\setminus \cF(\mathbb{F}_q)$. Take a point $P(v,t)\in\cF(\fq)\setminus \cF(\mathbb{F}_q)$ and suppose that it is fixed by a non-trivial element $\theta\in\Psi$. Then there exist $\alpha,\beta,\gamma,\delta\in \mathbb{F}_q$ with $\alpha\delta-\beta\gamma\neq 0$ such that $$\theta(v)=\frac{\alpha v +\beta}{\gamma v + \delta}=v, \quad \theta(t)=\frac{\alpha t +\beta}{\gamma t + \delta}=t.$$
 Since $v\neq t$ and $v,t\not\in \mathbb{F}_q$, this yields $v=t^q$. From (\ref{eqA1lug2017}), either $v^q=v$ or $t^q=t$, a contradiction. Since $|\mathfrak{G}_1|=q^3-q=|\cF(\fq)\setminus \cF(\mathbb{F}_q)|$, the claim follows.
\end{proof}

The next step is to construct a group from $\Psi$ that acts sharply transitively on $\cG_1$. For this purpose, define $\Gamma$ to be the group generated by the following maps on the generator set $G_1$:
\begin{itemize}
\item[(i)] $T_\alpha:(u,v,s,t)\rightarrow (u,v+\alpha,s,t+\alpha)$ with $\alpha\in\mathbb{F}_q;$
\item[(ii)]  $M_\rho:(u,v,s,t)\rightarrow (\rho u,\rho v, \rho s,\rho t)$ with $\rho$ square in $\mathbb{F}_q$;
\item[(iii)] $N_\sigma:(u,v,s,t)\rightarrow((\sigma v^{-2}u,v^{-1},\sigma t^{-2}s,t^{-1})$ with $\sigma^{(q+1)/2}=-1$. 
\item[(iv)] $R_\mu:(u,v,s,t)\rightarrow((\mu v^{-2}u,\mu v^{-1},\mu t^{-2}s,\mu t^{-1})$ with $\mu$ non-square in $\mathbb{F}_q$;
\item[(v)]  $L_\lambda:(u,v,s,t)\rightarrow (\lambda u,v,\lambda s,t),$ with $\lambda^{(q+1)/2}=1$ and $\lambda \in \fq$.
\end{itemize}
A straightforward computation shows that they indeed preserve $\cG_1$. To show that $\Gamma$ is sharply transitive on $\cG_1$, we first show that it is transitive. Take two generators in $\cG_1$, that is, two quadruples $(u,v,s,t)$ and $(u_1,v_1,s_1,t_1)$. Using Lemma \ref{lem1Alug1}, we find $\gamma \in \Gamma$ such that
$$\gamma:\,(u,v,s,t)\rightarrow (u_2,v_1,s_2,t_1)$$ with $$u_2^{(q+1)/2}=v_1^q-v_1,\,\, -s_2^{(q+1)/2}=t_1^q-t_1, \,\, s_2=\frac{(t_1-v_1^q)^2}{u_2^q}\,.$$
On the other hand,
$$u_1^{(q+1)/2}=v_1^q-v_1,\,\,-s_1^{(q+1)/2}=t_1^q-t_1,\,\,s_1=\frac{(t_1-v_1^q)^2}{u_1^q}\,.$$
Let $\lambda=u_1u_2^{-1}$. Since $s_1s_2^{-1}=u_2^qu_1^{-q}=u_2u_1^{-1}$, then
$$L_\lambda\circ\gamma: (u,v,s,t)\rightarrow (u_1,v_1,s_1,t_1).$$
To prove that the transitive action of $\Gamma$ on $\cG_1$ is also sharp, it is enough to show that $|\Gamma|=|\cG_1|$. Since $|\PGL(2,q)|=q^3-q$, the claim $|\Gamma|=|\cG_1|$ is a consequence of the following result.
\begin{lemma}
\label{lem2lug}  The subgroup of $\Gamma$ fixing both the second and fourth components, $v$ and $t$, coincides with the center $Z(\Gamma)$ of\, $\Gamma$. Also, $\Gamma/Z(\Gamma)\cong \PGL(2,q)$ and $Z(\Gamma)$ is a cyclic group of order $\ha(q+1)$.
\end{lemma}
\begin{proof} Let $L=\{L_\lambda|\lambda^{(q+1)/2}=1\}$. From the definition of its generators, $\Gamma$ induces a permutation group $\bar{\Gamma}$ on the pairs $(v,t)$ whose nucleus contains $L$, and $L$ is a subgroup of order $\ha (q+1)$ contained in $Z(\Gamma)$. On the other hand, if
$\tau:\,(u,v,s,t)\rightarrow (u_1,v,s_1,t)$ is an element in that nucleus then
$u_1^{(q+1)/2}=u^{(q+1)/2}$. Let $\lambda=u_1u^{-1}$. Then $\lambda^{(q+1)/2}=1$, and hence $\lambda^{q+1}=1$. From
$$s_1=\frac{(t-v^q)^2 }{u_1^q}= \frac{(t-v^q)^2 }{\lambda^q u^q}=\frac{1}{\lambda^q}s=\lambda s,$$
Hence $\tau=L_\lambda$. Thus, the nucleus coincides with $L$. To show that $\bar{\Gamma}\cong \PGL(2,q)$, recall that $\PGL(2,q)$ acts in its $3$-transitive permutation representation as the group of the invertible linear fractional maps with coefficients in $\mathbb{F}_q$. The maps (i), (ii), (iii), (iv) restricted on $v$ (equivalently on the pairs $(v,t)$) give a set of generators for $\PGL(2,q)$. Therefore,
$\bar{\Gamma}\cong \PGL(2,q)$.
\end{proof}
It is not true that $\Gamma$ itself contains a subgroup isomorphic to $\PGL(2,q)$. Nevertheless, as it follows from the forthcoming Lemma \ref{lem5lug2017},
its (unique) index $2$ subgroup $\Phi$  of $\Gamma$ is a splitting extension of $\PSL(2,q)$ by its center, that is, $\Phi\cong \PSL(2,q)\times C_{(q+1)/2}$ with a cyclic group of order $\ha(q+1)$.
Here $\Phi$ is generated by (i),(ii),(iii) and (v). From Lemma \ref{lem1Alug1}, $\Phi$ has two orbits on $\cG_1$, namely $\cM_1$ and $\cM_1'$. Let
 \begin{itemize}
\item[(vi)] $W:(u,v,s,t)\rightarrow (-u,t,-s,v).$
\end{itemize}
\begin{lemma}
\label{lem19lug2017} $W$ interchanges $\cM_1$ and $\cM_1'$.
\end{lemma}
\begin{proof}
Since $W$ commutes with $\Gamma$ and $\Phi$ is an index $2$ subgroup of $\Gamma$ whose orbits are $\cM_1$ and $\cM_1'$, either $W$ preserves both $\cM_1$ and $\cM_1'$ or interchanges them.

For $g\in \Gamma$ and $(u,v,s,t)\in \cG_1$, let $g(u,v,s,t)=(u_1,v_1,s_1,t_1)$. From \ref{eq19lug2017}, there exists a linear fractional map $f$ such that
\begin{equation}
\label{eq19lug2017}
v_1=\frac{\alpha v +\beta}{\gamma v + \delta}, \quad t_1=\frac{\alpha t +\beta}{\gamma t + \delta},
\end{equation}
with $\alpha,\beta,\gamma,\delta \in\mathbb{F}_q$ and $\alpha\delta-\beta\gamma\neq 0$. If $W$ and $g$ have the same action on $(u,v,s,t)$ then $v_1=t$ and $t_1=v_1$ hold. Therefore,
\begin{equation}
\label{eq21lug2017}
t=\frac{\alpha v +\beta}{\gamma v + \delta}, \quad v=\frac{\alpha t +\beta}{\gamma t + \delta}.
\end{equation}
We show first that $f$ has order $2$. From (\ref{eq21lug2017}), $\alpha v+\beta=\gamma t v +\delta t$ and
$\alpha t+\beta=\gamma t v +\delta v.$ From this $\alpha(v-t)=\delta(t-v)$. Since $F(x,x)\neq 0$ for $x \in \fq\setminus \mathbb{F}_q$, this yields $\delta=-\alpha$. Hence $f$ has order $2$.

We show that $f\in \PGL(2,q)\setminus \PSL(2,q)$. By way of contradiction, let $f\in \PSL(2,q)$. Then, up to conjugacy in $\PGL(2,q)$, we have $\alpha=-1,\beta=\delta=1, \gamma=0$. Then $v=-t+1$ and hence $F(t,-t+1)=0$. A straightforward computation shows that this yields $$((t-\ha)^2)^q+(t-\ha)^2=0.$$
Let $w=t-\ha$. Since $t\not\in \mathbb{F}_q$, we have
${(w^2)}^{q-1}=-1,$ whence
$$w^{(q-1)(q+1)}=(-1)^{(q+1)/2}.$$
Since $\ha (q+1)$ is odd, this yields $w^{q^2-1}=-1$. Therefore, $w\not\in \fq$ and hence $t\not\in \fq$, a contradiction.
\end{proof}

\subsubsection{Collineation groups realizing $\Gamma$ and $\Phi$}
\label{subcol}
Now we describe $\Gamma$ and $\Phi$ in terms of linear collineation groups.
\begin{lemma}
\label{lem4lug2017} The group $\PGU(4,q)$ has a subgroup $\mathfrak{G}$ with the following properties:
\begin{itemize}
\item[(i)] $\mathfrak{G}$ is an automorphism group of $\cX^+$ and $\cX^-$;
\item[(ii)] $\mathfrak{G}$ preserves the point-sets $\Delta^+,\Delta^-,\Omega$ and the generator set $\cG_1$;
\item[(iii)] $\mathfrak{G}$  acts faithfully on $\Delta^+$, $\Delta^-$ and $\cG_1$;
\item[(iv)] the permutation representation of $\mathfrak{G}$ on $\cG_1$ is  $\Gamma$; in particular $\mathfrak{G}\cong \Gamma$.
\item[(v)] $\mathfrak{G}$ acts on $\Omega$ as $\PGL(2,q)$ in its natural $3$-transitive permutation representation.
\item[(vi)] The collineation group induced by $\mathfrak{G}$ on $\pi$ is $\mathfrak{G}/Z(\mathfrak{G})\cong PGL(2,q)$ with $Z(\mathfrak{G})\cong C_{(q+1)/2}$.
 \end{itemize}
\end{lemma}
\begin{proof}
For any $\alpha\in \mathbb{F}_q$, and for any square $\rho$ in $\mathbb{F}_q^*$
let $$ {\mathbf{T}}_{a}= \left(
  \begin{array}{cccc}
    1   & 0 & 0  & 0 \\
    0   & 1 & 0  & 0 \\
    \alpha   & 0 & 1  & 0 \\
    \alpha^2 & 0 & 2\alpha & 1 \\
  \end{array}\right); \qquad
{\mathbf{M}}_{\rho}=\left(
  \begin{array}{cccc}
    \rho^{-1} & 0 & 0 & 0 \\
    0 & 1 & 0 & 0 \\
    0 & 0 & 1 & 0 \\
    0 & 0 & 0 & \rho
  \end{array}\right).
  $$
  For any  $\lambda\in \fq$ with $\lambda^{(q+1)/2}=1$, for any $\sigma\in \fq$ with $\sigma^{(q+1)/2}=-1$, and for any non-square $\mu$ in $\mathbb{F}_q$
  $$
{\mathbf{L}}_{\lambda}=\left(
  \begin{array}{cccc}
    1 & 0 & 0 & 0 \\
    0 & \lambda & 0 & 0 \\
    0 & 0 & 1 & 0 \\
    0 & 0 & 0 & 1
  \end{array}\right); \qquad
{\mathbf{N}}_\sigma=\left(\begin{array}{cccc}
    0 & 0 & 0 &1  \\
    0 & \sigma & 0 &0  \\
    0 & 0 & 1 &0  \\
    1 & 0 & 0 &0
  \end{array}\right); \qquad
{\mathbf{R}}_{\mu}=\left(
  \begin{array}{cccc}
    0 & 0 & 0 & \mu^{-1} \\
    0 & 1 & 0 & 0 \\
    0 & 0 & 1 & 0 \\
    \mu & 0 & 0 & 0
  \end{array}\right).
  $$
A straightforward computation shows that following facts on the linear collineations of $\PGU(4,q)$ associated with the above matrices: They preserve the curves $\cX^+$, $\cX^-$, the point-sets  $\Delta^+$, $\Delta^-,\Omega$, and the generator set $\cG_1$; their actions on $\cG_1$ coincide with the actions of the corresponding maps (i), (ii), (iii), (iv), (v), defined in Subsection \ref{sub1}. Let $\mathfrak{G}$ denote the subgroup of $\PGU(4,q)$ generated by all these collineations.  Assume that a non-trivial element $\mathfrak{g}\in \mathfrak{G}$ fixes $\Delta^+$ pointwise. Then $\Delta^+$ is contained in a plane of $\PG(3,\fq)$. Since the curve $\cX^+$ has degree $q+1$, this yields that $|\Delta^+|\leq \ha(q+1)$ contradicting $|\Delta^+|=\ha(q^3-q)$. Similarly for $\Delta^-$. Finally, assume that  a non-trivial element $\mathfrak{g}\in \mathfrak{G}$ fixes each generator in $\cG_1$. Since through every point in $\Delta^+$ there exist more than one generators of $\cG_1$, this yields that $\mathfrak{g}$ fixes $\Delta^+$ pointwise; a contradiction.

 To show (v) look at the action of $\mathfrak{G}$ on the plane $\pi$ of equation $X_1=0$. If $\mathfrak{N}$ is the nucleus of the permutation representation of $\mathfrak{G}$ on $\pi$ then $\mathfrak{N}$ contains the cyclic group $\mathfrak{L}$ of order $\ha(q+1)$ consisting of the linear collineations associated with the matrices $L_\lambda$. Observe that $\mathfrak{N}$ is the (cyclic) homology group of $\cU_3$ with center $X_\infty$ and axis $\pi$. Also,
 $\mathfrak{L}$ is contained in $\mathfrak{G}$. Furthermore, $\overline{\mathfrak{G}}=\mathfrak{G}/\mathfrak{N}$ is a linear collineation group of $\pi$ preserving the conic $\cC$ of equation $X_0X_3-X_2^2=0$. Since $\Omega$ is the set of all $\mathbb{F}_q$-rational points of $\cC$ and $\overline{\mathfrak{G}}$ preserves $\Omega$, it turns out that $\overline{\mathfrak{G}}$ is a linear collineation group of $PG(2,q)$. We show that this yields $\overline{\mathfrak{G}}\cong \PGL(2,q)$ and $\mathfrak{N}=\mathfrak{L}$. On one hand, the factor group $\mathfrak{G}/\mathfrak{L}$ contains $\mathfrak{N}/\mathfrak{L}$ as a cyclic normal subgroup. On the other hand $\mathfrak{G}\cong\Gamma$ by (iii) and $\mathfrak{G}/Z(\mathfrak{G})\cong \PGL(2,q)$ with $Z(\mathfrak{G})=\mathfrak{L}$ by Lemma \ref{lem2lug}. Since $\PGL(2,q)$ has no cyclic normal subgroup, this proves indeed that  $\mathfrak{N}=\mathfrak{L}$ with $Z(\mathfrak{G})=\mathfrak{N}$ whence $\mathfrak{G}/Z(\mathfrak{G})\cong \PGL(2,q)$ also follows.
 \end{proof}
\begin{lemma}
\label{lem5lug2017} $\mathfrak{G}$ has an index $2$  subgroup $\mathfrak{H}$ isomorphic to $\PSL(2,q)\times C_{(q+1)/2}$. In the isomorphism $\mathfrak{G}\cong \Gamma$, $\mathfrak{H}$ and $\Phi$ correspond.
\end{lemma}
\begin{proof}  From (v) of Lemma \ref{lem4lug2017}, the $1$-point stabilizer of $\mathfrak{G}$ in its action on $\Omega$
has a normal $p$-subgroup acting on the remaining $q$ points of $\Omega$ as a sharply transitive permutation group. Let $p(\mathfrak{G})$ be the normal subgroup of $\mathfrak{G}$ generated by its Sylow $p$-subgroups. From \cite[Theorem 2.4]{hering}, either $p(\mathfrak{G})$ acts faithfully on $\Omega$ and $p(\mathfrak{G})\cong \PSL(2,q)$, or $p(\mathfrak{G})$ contains an element $\mathfrak{g}$ of order $2$  fixing $\Omega$ pointwise. In the latter case, $\mathfrak{g}$ would leave each line through $X_\infty$ invariant. As some of these lines meets $\cX^+$ in $\ha (q+1)$ points distinct from $X_\infty$ and those in $\pi$, this would imply that $\ha(q+1)$ is an even number, a contradiction. Hence, $p(\mathfrak{G})\cong \PSL(2,q)$. Since $\mathfrak{L}$ commutes with $p(\mathfrak{G})$, it turns out that $\mathfrak{H}=p(\mathfrak{G})\times \mathfrak{L}\cong \PSL(2,q)\times C_{(q+1)/2}$ is a subgroup of $\mathfrak{G}$ of index $2$.
\end{proof}
We end this section with a result about elements of order $2$  in $\mathfrak{G}$ that will be used in Section \ref{secE}. A linear collineation of order $2$ of $\PGU(4,q)$, has always some fixed points. More precisely, two cases can occur for the set $\Sigma$ of fixed points of an involution $\sigma$ in $\PG(3,\fq)$: either $\Sigma$ consists of all points of a plane (axis of $\sigma$) together with a further point (center of $\sigma)$ and each line through the center is preserved by $\sigma$, or $\Sigma$ consists of two skew lines (axes) which are conjugate with respect the polarity associated to $\cU_3$. In the former case, $\sigma$ is a {\emph{homology}}, in the latter one is a {\emph{skew perspectivity}}. The linear collineation associated to $\mathbf{R}_\mu$ is a homology with center $(1,0,0,-\mu^{-1})$ and axis $X_0=\mu X_3$ while both $\mathbf{N}_{-1}$ and $\mathbf{M}_{-1}$ are matrices associated to skew perspectivities. In the former case, the axes are the lines $X_0=X_3,\,X_1=0$ and $X_0=-X_3,\,X_2=0$, in the latter case the axes are $X_0=0,\,X_3=0$ and $X_1=0,X_2=0$. Also, the product $\mathbf{N}_1\mathbf{M}_{-1}$ has order $2$ and it is a matrix associated to
the skew perspectivity with axes $X_0=X_3,\,X_2=0$ and $X_0=-X_3,X_1=0$.
\begin{lemma}
\label{lem7lug2017} The elements of order $2$ in $\mathfrak{H}$ are skew perspectivities while those in $\mathfrak{G}\setminus \mathfrak{H}$ are homologies.
\end{lemma}
\begin{proof} The group $\PGL(2,q)$ has two conjugacy classes of elements of order $2$. One class consists of elements in $\PSL(2,q)$ and they are even permutations. Hence the elements of order $2$ in $\PGL(2,q)\setminus \PSL(2,q)$ form the other class and they are odd permutations. Since $Z(\mathfrak{G})$ has odd order, (iv) of Lemma \ref{lem4lug2017} together with Lemma \ref{lem2lug} show that this picture does not change when the conjugacy classes of elements of order $2$ in $\mathfrak{G}$ and  $\mathfrak{H}$ are considered. As both $\mathbf{N}_{1}$ and $\mathbf{M}_{-1}$, as well as their product are matrices associated to skew perspectivites, they are in $\mathfrak{H}$. Since $\mathbf{R}_\mu$ is a homology, this also shows that $\mathbf{R}_\mu\in\mathfrak{G}\setminus \mathfrak{H}$.
\end{proof}
\begin{remark}
\label{rem26lug} {\em{
A straightforward computation shows that the linear collineation $\mathfrak{w}\in \PGU(4,q)$ associated with the matrix
$$
{\mathbf{W}}=\left(
  \begin{array}{cccc}
    1 & 0 & 0 & 0 \\
    0 & -1 & 0 & 0 \\
    0 & 0 & 1 & 0 \\
    0 & 0 & 0 & 1
  \end{array}\right)
$$
preserves $\cG_1$, and it acts on $\cG_1$ as the map $W$.
It should be noticed however that $\mathfrak{w}$ interchanges $\cX^+$ and $\cX^-$.  Also, $\mathfrak{w}$ preserves $\Omega$, has order $2$ and commutes with $\mathfrak{G}$. The linear group generated by $\mathfrak{G}$ and $\mathfrak{w}$ is the direct product $\mathfrak{G}\times \langle\mathfrak{w}\rangle$.}}
\end{remark}
\subsection{Conditions (C) and (D) on $\Omega$}
\label{subome}
Let $\cG_2$ be the set of all generators meeting $\Omega$.
\begin{lemma}
\label{lemA5lug2017} $\mathfrak{G}$ acts transitively on $\cG_2$ while $\mathfrak{H}$ has two orbits on $\cG_2$.
\end{lemma}
\begin{proof} Lemma \ref{lem2lug} together with the second claim in (iv) of Lemma \ref{lem4lug2017} show that the $1$-point stabilizer of $\mathfrak{G}$ in $\Omega$ has order $\ha (q+1)q(q-1)$.
Since $Z_\infty=(0,0,0,1) \in\Omega$ and $\mathfrak{H}$ is also transitive on $\Omega$, we may limit ourselves to the generators through $Z_\infty$. They are the $q+1$ lines with parametric equations
\begin{equation}
\label{eqA1agosto}
 g_k=\,
\begin{cases}
X_0=0,\\
X_1=k,\\
X_2=1,\\
X_3=T
\end{cases}
\end{equation}
where $k\in \fq$ runs over the $q+1$ roots of the polynomial $X^{q+1}+2$. A straightforward computation shows the following facts:
\begin{itemize}
\item[(A)] $\mathbf{T}_\alpha$ and $\mathbf{M}_\rho$ with $\rho$ a non-zero square in $\mathbb{F}_q$ are matrices of linear collineations in $\mathfrak{H}$ fixing $Z_\infty$, preserving each $g_k$ and generating a subgroup of order $\ha q(q-1)$ contained in $\mathfrak{H}$. This subgroup has three non-trivial orbits on $g_k$. One consists of the $q-1$ points $U_T=(0,k,1,T)$ with $T\in \mathbb{F}_q^*$, the other two orbits consist each of $\ha(q^2-q)$ points $U_T$. More precisely, for a fixed $b\in \fq\setminus \mathbb{F}_q$, the point $U_T$ with $T=\alpha+\beta b$ and $\alpha,\beta \in \mathbb{F}_q$ falls into one or the other orbit according as $\beta$ is a non-zero square or a non-square element in $\mathbb{F}_q$. The latter two orbits are fused in a unique orbit under the action $\mathfrak{w}$. Therefore, the subgroup of $\mathfrak{G}\times \langle \mathfrak{w} \rangle$ fixing each $g_k$ has two nontrivial orbits on $g_k$.
\item[(B)] The linear collineation associated to the matrix $\mathbf{N}_{-1}\mathbf{R}_\mu$ also fixes $Z_\infty$ but it takes each $g_k$ to $g_{-k}$ and is contained in $\mathfrak{G}\setminus \mathfrak{H}$.
\item[(C)] The linear collineation associated to the matrix $\mathbf{L}_{\lambda}$ with $\lambda^{(q+1)/2}=1$ is in $\mathfrak{H}$, fixes $Z_\infty$ and takes each $g_k$ to $g_{\lambda k}$.
\item[(D)] The linear collineation associated to the matrix $\mathbf{W}$ fixes $Z_\infty$ and takes each $g_k$ to $g_{-k}$.
\end{itemize}
Therefore $\mathfrak{G}_{Z_\infty}$ induces a sharply transitive permutation group on the generator set through $Z_\infty$ while its subgroup $\mathfrak{H}_{Z_\infty}$ has two orbits.
\end{proof}
\begin{remark}
\label{rem2agos2017} {\em{The proof of Lemma \ref{lemA5lug2017} also shows that the subgroup of $\mathfrak{G}$ fixing $Z_\infty$ has three non-trivial orbits on the point-set cut out from $\cU_3$ by the tangent plane $\pi_{Z_\infty}$ to $\cU_3$ at $Z_{\infty}$: One orbit consists of $q^2+q$ points in $\PG(3,\mathbb{F}_q)$ while each of the other two orbits consists of $\ha(q^3-q)$ points in $\PG(3,\fq)\setminus \PG(3,\mathbb{F}_q)$ and $\mathfrak{G}$ acts faithfully on both as sharply transitive permutation group. The latter two orbits are fused in a unique one under the action of the subgroup of $\mathfrak{G}\times \langle\mathfrak{w}\rangle$ which fixes $Z_\infty$. }}
\end{remark}

\subsection{Conditions (C) and (D) on $\fq(\cX^+)$.}
Since $\Gamma$ has exactly two orbits on $\cG$, namely $\cG_1$ and $\cG_2$, the results obtained in Subsections \ref{sub1}, \ref{subcol} and \ref{subome} have the following corollary.
\begin{theorem}
\label{teo6lug2017} Conditions $\rm{(C)}$ and $\rm{(D)}$ are fulfilled for $\cX=\cX^+$,  $\Gamma=\mathfrak{G}$ and $\Phi=\mathfrak{H}$.
\end{theorem}
\begin{remark}
\label{remA3agos} {\em{
 $\cG=\cG_1\cup \cG_2$ with $\cG_1=\cM_1\cup \cM_1'$ and $\cG_2=\cM_2\cup \cM_2'$ where $\cG_1,\cG_2$ are the $\mathfrak{G}$-orbits on $\cG$ whereas $\cM_1,\cM_1',\cM_2,\cM_2'$ are $\mathfrak{H}$-orbits
on $\cG_1$ and $\cG_2$ respectively. This notation fits with Section \ref{hema}.}}
\end{remark}
\section{Points satisfying Condition (E)}
\label{secE}
It is useful to look at the plane $\pi$ of equation $X_1=0$ as the projective plane $\PG(2,q^2)$ with homogenous coordinates $(X_0,X_2,X_3)$. Then $\cC$ is the conic of equation $X_0X_3-X_2^2=0$, and $\Omega$ is the set of points of $\cC$ lying in the (canonical Baer) subplane $\PG(2,q)$. Any line of $\PG(2,q)$, viewed as a line in $\PG(2,q^2)$, has $q^2-q$ points outside $\PG(2,q)$ and each such point is incident with a unique line of $PG(2,q)$. Therefore the points in  $\PG(2,q^2)\setminus \PG(2,q)$ are of three types with respect to lines of $PG(2,q)$, namely
\begin{itemize}
\item[(I)] point on a (unique) line disjoint from $\Omega$ which meets $\cC$ in two (distinct) points both in $\PG(2,q^2)\setminus \PG(2,q)$;
\item[(II)] point on a (unique) line meeting $\Omega$ in two distinct points;
\item[(III)] point on a (unique) line which is tangent to $\cC$ with tangency point in $\Omega$;
\end{itemize}
Furthermore, the linear collineation group $\Lambda$ of $\PG(2,q)$ preserving $\Omega$ (and $\cC$ when viewed as a collineation group  of $\PG(2,q^2)$) is isomorphic to $\PGL(2,q)$ and acts on $\Omega$ as $\PGL(2,q)$ in its $3$-transitive permutation representation. Let $\Psi\cong \PSL(2,q)$ be the unique index $2$ subgroup of $\Lambda$. From (vi) of Lemma \ref{lem4lug2017}, $\Lambda$ may be identified with the quotient group $\mathfrak{G}/Z(\mathfrak{L})$ so that $\Psi$ coincides with $\mathfrak{H}/Z(\mathfrak{L})$.
Take a line $\ell$ of $\PG(2,q)$ viewed as a line of $\PG(2,q^2)$.

\subsection{Case (I)} Suppose $\ell$ to meet $\cC$ in two distinct points in $\PG(2,q^2)\setminus \PG(2,q)$. Obviously, $\ell$ is an external line to $\Omega$. Since $q\equiv 1 \pmod 4$, $\Lambda\setminus \Psi$ contains an element $g$ of order $2$ that fixes $\ell$ pointwise. Therefore, $\mathfrak{G}$ contains
a subgroup of order $q+1$ containing $\mathfrak{L}$ but not contained in $\mathfrak{H}$. Since $\mathfrak{L}=Z(\mathfrak{G})$, such a subgroup contains a (unique) element $\mathfrak{g}\in \mathfrak{G}\setminus \mathfrak{H}$ of order $2$ that acts on $\PG(2,q^2)$ as $g$.
From Lemma \ref{lem7lug2017}, $\mathfrak{g}$ is a homology. The axis of $\mathfrak{g}$ is the plane through $\ell$ and $X_\infty$ while its center is the pole of $\ell$ with respect to the orthogonal polarity associated with $\cC$.  Therefore, if $P'$ is a point of type (I), then $\mathfrak{g}$ fixes each point $P$ of the line $r$ through $P'$ and $X_\infty$. We also have that $\mathfrak{g}$ interchanges $\cM_1$ and $\cM_1'$.

\subsection{Case (II)} Suppose $\ell$ to meet $\Omega$ in two distinct points. Since $q\equiv 1 \pmod 4$, $\Psi$ contains an element $g$ of order $2$ that fixes $\ell$ pointwise. Therefore, $\mathfrak{H}$ contains a subgroup of order $q+1$. By $\mathfrak{L}=Z(\mathfrak{G})$, such a subgroup contains a (unique) element $\mathfrak{g}\in \mathfrak{H}$ of order $2$ that acts on $\PG(2,q^2)$ as $g$. From Lemma \ref{lem7lug2017}, $\mathfrak{g}$ is a skew perspectivity. One of the lines consisting of fixed points of $\mathfrak{g}$ is $\ell$, the other is the line $r$ through $X_\infty$ and the pole of $\ell$ with respect to orthogonal polarity associated with $\cC$. Therefore, if $P'$ is a point of type (II), then $P'$ and $X_\infty$ are the unique fixed points of $\mathfrak{g}$ on the line $r$. The linear collineation $\mathfrak{w}$ associated with $W$ also preserves $r$ and fixes both $P'$ and $X_\infty$, as $\mathfrak{w}$ is a homology with axis $\pi$ and center $X_\infty$. The collineation group $\mathfrak{S}$ generated by $\mathfrak{g}$ and $\mathfrak{w}$ is an elementary abelian group of order $4$. Since $|r\cap \cU_3|=q+1$ but $X_\infty,P'\not\in \cU_3$, our hypothesis
$q\equiv 1 \pmod 4$ yields that a nontrivial element $\mathfrak{s}\in\mathfrak{S}$ fixes a point in $r\cap \cU_3$ and hence $\mathfrak{s}$ fixes each point $P$ of $r$. Since $\mathfrak{s}\neq \mathfrak{w}$, we also have that $\mathfrak{s}$ interchanges $\cM_1$ and $\cM_1'$.

\subsection{Case (III)} Suppose $\ell$ to be a tangent to $\cC$ (at a point in $\Omega$). Let $P'$ be a point of $\ell$ in $\PG(2,q^2)\setminus \PG(2,q)$, and let $P\in \cU_3$ be a point whose projection from $X_\infty$ is $P'$. Then the stabilizer of $P'$ in $\Lambda$ is trivial. Therefore, no element $\mathfrak{G}$ fixing $P$ can interchange $\cM_1$ and $\cM_1'$. Furthermore,  if the tangency point is $Z_\infty$, it follows from claim (A) in the proof of Lemma \ref{lemA5lug2017} that the points $P\in \cU_3$ whose projection $P'$ lies on $\ell$ form two orbits under the action of $\mathbb{G}\times\langle \mathfrak{w}\rangle$. One orbit consists of all points with $P'\in \PG(2,\mathbb{F}_q)$  while those with $P'\in \PG(2,\fq)\setminus \PG(2,\mathbb{F}_q)$ form the other orbit. As $\mathfrak{G}$ is transitive one $\Omega$ this holds true for any tangency point, that is, for any tangent to $\cC$ at a point in $\Omega$.
\begin{theorem}
\label{em24lug2017} If the projection of a point $P\in \cU_3$ on $\pi$ is a point $P'$ of type $\rm{(I)}$ or $\rm{(II)}$, or $P'\in \PG(2,q)$  then Condition $\rm{(E)}$ is fulfilled  for $\cX=\cX^+$,  $\Gamma=\mathfrak{G}$ and $\Phi=\mathfrak{H}$.
\end{theorem}
\begin{proof} From the above discussion, we may limit ourselves to the case $P'\in \PG(2,q)$. Since $P'\not\in \Omega$, there exists an external line $\ell$ to $\Omega$ which passes through $P'$. Arguing as in Case (I) shows that if $P$ is any point on the line through $P'$ and $X_\infty$ then $P$ is fixed by an element $\mathfrak{g}\in \mathfrak{G}$ which
interchanges $\cM_1$ and $\cM_1'$.
\end{proof}

\section{Condition (B) for Case (III)}
From a computer aided investigation performed for $5\leq q\leq 101$, it turns out that Condition (B) is rarely satisfied in Case (III), that is, for points $P$ whose projection from $X_\infty$ on $\pi$ is a point $P'$ lying on a tangent $\ell$ to $\cC$. On the other hand, finding some (possibly infinite) values of $q$ for which this does occur is strictly necessary for our construction to produce  a hemisystem of the Hermitian surface of $\PG(3,\fq).$ Here we focus on those values of $q$ which satisfy the hypothesis
\begin{equation}
\label{eqE24lug} {\mbox{$p\equiv  1 \pmod 8$ or $q$ is a an even power of $p$}},
\end{equation}
and prove the following
\begin{theorem}
\label{teo5ago2017B} Condition $\rm{(B)}$ for Case $\rm{(III)}$ is satisfied if and only if the number $N_q$ of $\mathbb{F}_q$-rational points of the elliptic curve
with affine equation $Y^2=X^3-X$ equals either $q-1$, or $q+3$.
\end{theorem}
\begin{remark}
\label{rem1ago}
{\em{Since $q\equiv 1 \pmod 4$ is assumed in this paper, (\ref{eqE24lug}) holds if and only if $2$ is a square in $\mathbb{F}_q$.}}
\end{remark}
To prove Theorem \ref{teo5ago2017B} the first step is to prove the following result.
\begin{proposition}
\label{teo5ago2017} Condition $\rm{(B)}$ for Case $\rm{(III)}$ is satisfied if and only if the number $n_q$ of $\xi\in \mathbb{F}_q$ for which
$f(\xi)=\xi^4-24\omega \xi^2+16\omega^2$ square in $\mathbb{F}_q$  equals either $\ha(q+1)$, or $\ha(q-3)$.
\end{proposition}

Theorem \ref{teo5ago2017} is obtained as a corollary of several lemmas proven in Subsections below. Their proofs require some group theory and longer computation.

According to Remark (\ref{rem1ago}), fix one of the square roots of $2$ and denote it by $\sqrt{2}$. Then the other root is $-\sqrt{2}$. Furthermore, $2-\sqrt{2}$, $10-7\sqrt{2}$ are simultaneously squares or non-squares in $\mathbb{F}_q$.

Since $\mathfrak{G}$ is transitive on $\Omega$, the point $O=(1,0,0,0)$ may be assumed to be the tangency point of $\ell$. Then $\ell$ has equation $X_1=0,\,X_3=0$, and  $P=(a_0,a_1,a_2,0)$ with $a_1\neq 0$ and $a_1^{q+1}+2a_2^{q+1}=0$.

If $a_0=0$ then $P=(0,d,1,0)$ with $d^{q+1}+2=0$, and hence $P'=(0,0,1,0)$ which is a point in $\PG(2,\mathbb{F}_q)$. By Theorem \ref{em24lug2017} the case $a_0=0$ can be dismissed.
Therefore $a_0=1$ may be assumed.

Therefore, non-homogeneous coordinates $X=X_1/X_0$, $Y=X_2/X_0,Z=X_3/X_0$ can be used so that $P=(a,b,0)$ where $a^{q+1}+2b^{q+1}=0$.
Since the latter equation holds for $a=\pm \sqrt{-2}b$, two obvious choices arise for $P$. Eventually, they are useful and we also have the advantage that computation may simultaneously be carried out. This motives the following notation.
$${\mbox{$P=(\varepsilon \sqrt{-2}\,b,b,0)$, where $\varepsilon\in \{1,-1\}$.}}$$
So far in our discussion, $b$ may be any element from $\fq\in \mathbb{F}_q$. For the sake of simplicity in computation, it is useful to fix $b$ so that
\begin{equation}
\label{eqB24lug}
{\rm{Tr}}(b)=b^q+b=0.
\end{equation}
 Then
\begin{equation}
\label{eqH24lug}
\omega=b^2=-b^{q+1}\in \mathbb{F}_q
\end{equation}
Conversely, for any non-square $\omega$ in $\mathbb{F}_q$, there exists $b$ satisfying (\ref{eqB24lug}) and (\ref{eqH24lug}).
Furthermore, let
\begin{equation}
\label{eqA7ago}
j=b^{(q-1)/2}.
\end{equation}
From (\ref{eqB24lug}), $b^{q-1}=-1$, whence $j^2=-1$. Since $q\equiv {1 \pmod 4}$, this yields $j\in\mathbb{F}_q$. Also, up to change of $\sqrt{2}$ with its opposite, $$\sqrt{-2}=j\sqrt{2}.$$

\subsection{Case of $\cG_1$}
\label{ssg1}
We keep our notation of $P=(u,v,v^2)=P_{u,v}$ for points in $\Delta^+$.
\begin{lemma}
\label{lem10ago} Let $v\in \fq\setminus \mathbb{F}_q$. Then there exists $u\in\fq$ such that the line joining $P$ to $P_{u,v}=P(u,v,v^2)\in\Delta^+$ is a generator of $\cU_3$ if and only if
\begin{equation}
\label{eqF24lug}
(v^2+2bv)^{(q+1)/2}=
\begin{cases}
{\mbox{$-\varepsilon\sqrt{2}\,b (v^q-v)$ for $\sqrt{2}$ square in $\mathbb{F}_q$}},\\
{\mbox{\,\,\,\,$\varepsilon\sqrt{2}\,b (v^q-v)$ for $\sqrt{2}$ non-square in $\mathbb{F}_q$}}.
\end{cases}
\end{equation}
If (\ref{eqF24lug}) holds then $u$ is uniquely determined by $v$.
\end{lemma}
\begin{proof} The line $\ell$ joining $P$ to a point $P(u,v,v^2)\in \cU_3$ is a generator if and only if the latter point lies on the tangent plane $\pi_P$ to $\cU_3$ at $P$. A direct computation shows that this occurs if and only if
\begin{equation}
\label{eqZ5agos}
u=
\begin{cases}
{\mbox{$-\frac{v^2+2bv}{\varepsilon\sqrt{-2}\,b}$ for $\sqrt{2}$ square in $\mathbb{F}_q$}},\\
{\mbox{$\,\,\,\,\frac{v^2+2bv}{\varepsilon\sqrt{-2}\,b}$ for $\sqrt{2}$ non-square in $\mathbb{F}_q$}}.
\end{cases}
\end{equation}
If $P(u,v,v^2)=P_{u,v}\in \Delta^+$ then $u^{(q+1)/2}=v^q-v$, and hence Equations (\ref{eqZ5agos}) and (\ref{eqA7ago}) imply (\ref{eqF24lug}). This shows that if $\ell$ is a generator then (\ref{eqF24lug}) holds. Conversely, let $u\in \fq$ be as
in (\ref{eqZ5agos}). Then the point $P(u,v,v^2)$ lies in $\pi_P$. Since  (\ref{eqF24lug}) implies $u^{(q+1)/2}=v^q-v$, we also have $P(u,v,v^2)\in \Delta^+$. Therefore,
 the line through $P$ and $P(u,v,v^2)$
 is a generator.
\end{proof}
Lemma \ref{lem10ago} can also be stated for $s,t$ provided that $u,v$ are replaced by $s,t$ and (\ref{eqF24lug}) and (\ref{eqZ5agos}) by
\begin{equation}
\label{eqZ5agosbis}
s=
\begin{cases}
{\mbox{$-\frac{t^2+2bt}{\varepsilon\sqrt{-2}\,b}$ for $\sqrt{2}$ square in $\mathbb{F}_q$}},\\
{\mbox{$\,\,\,\,\frac{t^2+2bt}{\varepsilon\sqrt{-2}\,b}$ for $\sqrt{2}$ non-square in $\mathbb{F}_q$}}.
\end{cases}
\end{equation}
and
\begin{equation}
\label{eqF24lugbis}
(t^2+2bt)^{(q+1)/2}=
\begin{cases}
{\mbox{$\varepsilon\sqrt{2}\,b (t^q-t)$ for $\sqrt{2}$ square in $\mathbb{F}_q$}},\\
{\mbox{$-\varepsilon\sqrt{2}\,b (t^q-t)$ for $\sqrt{2}$ non-square in $\mathbb{F}_q$}}.
\end{cases}
\end{equation}
Furthermore, the points $P_{u,v},Q_{s,t}$ and $P$ are collinear if and only if
\begin{equation}
\label{eqT5agos}
\begin{cases}
\,(t^2-v^2)\varepsilon \sqrt{-2}b=t^2u+v^2s.\\
\, vt-b(t+v)=0.
\end{cases}
\end{equation}
This shows the following lemma.
\begin{lemma}
\label{lemA5agos} Let $v,t\in \fq\setminus \mathbb{F}_q$ with $F(v,t)=0$. If the line joining $P_{u,v}\in \Delta^+$ and $Q_{s,t}\in \Delta^-$ is a generator through $P$ then
\begin{equation}
\label{eqD24lug}
vt-b(v+t)=0.
\end{equation}
\end{lemma}
\begin{remark} \emph{If $F(v,t)=0$ and (\ref{eqD24lug}) holds, choose $u,s$ as in (\ref{eqZ5agos}). Then the first Equation in (\ref{eqT5agos}) also holds. This can be shown by direct computation.
However, the third Equation in (\ref{eq25Alug}) may happen not be true. This shows that the converse of Lemma \ref{lemA5agos} may fail.}
\end{remark}
From now on, we write $\pm$ where $+$ and $-$ are taken according as $\sqrt{2}$ is square or non-square in $\mathbb{F}_q$. Then  (\ref{eqF24lug}) reads $-(v^2+2bv)^{(q+1)/2}=\pm \varepsilon\sqrt{2}\,b (v^q-v)$.

Lemma \ref{eqF24lug} shows that counting the solutions $v$ of Equation (\ref{eqD24lug}) gives the number $n_P$ of generators in $\cG_1$ which pass through $P$.

\begin{lemma}
\label{lem27lug} Equation {\rm{(\ref{eqF24lug})}} has exactly $\ha(q+1)$ solutions $v$ in $\fq\setminus \mathbb{F}_q$.
\end{lemma}
\begin{proof}
Regard $\fq$ as the quadratic extension $\mathbb{F}_q(b)$ of $\mathbb{F}_q$, and write $v=v_1+bv_2$ with $v_1,v_2\in \mathbb{F}_q$.
Let $r=b^{-1}v$. Then  Equation in (\ref{eqF24lug}) reads
\begin{equation}
\label{eqQ26lug}
(r^2+2r)^{(q+1)/2}=\pm\varepsilon\sqrt{2}(r^q+r).
\end{equation}
Since $\pm\varepsilon\sqrt{2}(r^q+r)\in \mathbb{F}_q$ this yields
$$(r^2+2r)^{(q^2-1)/2}=1.$$
Therefore $z=r^2+2r$ is a non-zero square element in $\fq$, and (\ref{eqQ26lug}) can be written as
\begin{equation}
\label{eqC27lug}
\begin{cases}
z^2=r^2+2r,\\
z^{q+1}=\pm\varepsilon\sqrt{2}(r^q+r).
\end{cases}
\end{equation}
Let $\lambda=zr^{-1}$. Then
\begin{equation}
\label{eqD27lug}
\begin{cases}
\lambda^2 r=r+2,\\
\lambda^{q+1}r^{q+1}=\pm\varepsilon\sqrt{2}(r^q+r).
\end{cases}
\end{equation}
Therefore $r=2(\lambda^2-1)^{-1}$. Eliminating $r$ from (\ref{eqD27lug}) gives
\begin{equation}
\label{eqCC27lug}
\lambda^{2q}\mp\varepsilon\sqrt{2}\lambda^{q+1}+\lambda^2-2=0.
\end{equation}
From this with $\lambda=\lambda_1+b\lambda_2$ where $\lambda_1,\lambda_2\in \mathbb{F}_q$,
\begin{equation}
\label{eq330lug}
(2\mp\varepsilon\sqrt{2})\lambda_1^2+\omega(2\pm\varepsilon\sqrt{2})\lambda_2^2-2=0.
\end{equation}
Since $(2-\varepsilon \sqrt{2})(2+\varepsilon\sqrt{2})=2$ and $\omega$ is non-square in $\mathbb{F}_q$, Equation (\ref{eq330lug}) is satisfied by exactly $q+1$ ordered pairs $(\lambda_1,\lambda_2)$ with $\lambda_1,\lambda_2\in \mathbb{F}_q$. Therefore,
$\ha(q+1)$ is the number of elements $r=2(\lambda^2-1)^{-1}$ in $\fq$. It remains to show that $br\not\in \mathbb{F}_q$. Suppose on the contrary that $(br)^q=br$. By  (\ref{eqH24lug}),
$r^q=-r$. From the second Equation in (\ref{eqD27lug}), $\lambda r=0$ which contradicts the first Equation in (\ref{eqD27lug}).
\end{proof}
\begin{lemma}
\label{lemA26lug} For each solution $v=v_1+bv_2\in \fq\setminus\mathbb{F}_q$ of Equation {\rm{(\ref{eqF24lug})}},
$$v_2^2+\frac{1}{\omega}v_1^2+2(1\pm\varepsilon \sqrt{2})v_2$$ is a square element in $\mathbb{F}_q$,
\end{lemma}
\begin{proof}
Write $r=r_1+br_2$ and $z=z_1+bz_2$. Now, (\ref{eqC27lug}) reads
$$
\begin{cases}
r_1^2+\omega r_2^2+2r_1=z_1^2+\omega z_2^2,\\
r_1r_2+r_2=z_1z_2,\\
\pm\varepsilon 2\sqrt{2}r_1=z_1^2-\omega z_2^2.
\end{cases}
$$
In terms of $v_1,v_2,z_1,z_2$,
\begin{equation}
\label{eq1126lug}
\begin{cases}
v_2^2+\frac{1}{\omega}v_1^2+2v_2=z_1^2+\omega z_2^2,\\
v_2v_1+v_1=\omega z_1z_2,\\
\pm\varepsilon 2\sqrt{2}v_2=z_1^2-\omega z_2^2.
\end{cases}
\end{equation}
Summing the first and the last Equations in (\ref{eq1126lug}) gives the result.
\end{proof}

Our next step is to show a generator in $\cG_1$ that passes through $P$.
For this purpose, let
$$\chi=
\begin{cases}
{\mbox{\,\,\,\,$1$ if either $2-\sqrt{2}$ is a square and $\varepsilon=-1$, or $2-\sqrt{2}$ is a non-square and $\varepsilon=1$}},\\
{\mbox{$-1$ if either $2-\sqrt{2}$ is a square and $\varepsilon=-1$, or $2-\sqrt{2}$ is a non-square and $\varepsilon=1$.}}
\end{cases}
$$
Then $\chi=\varepsilon(2-\chi\sqrt{2})^{(q-1)/2}$.
Now let
\begin{equation}
\label{eq526lug}
v_0=-2(1-\chi \sqrt{2})b,\quad u_0=\pm\frac{-4}{\varepsilon\sqrt{-2}}(2-\chi \sqrt{2})b.
\end{equation}
Then
$$v_0^q-v_0=4(1-\chi\sqrt{2})b=-\frac{4}{\sqrt{2}} \chi(2-\chi\sqrt{2})b.$$
On the other hand, as $\sqrt{-2}^{\,(q-1)/2}=\sqrt{2}^{\,(q-1)/2}=1$, or $-1$ according as $\sqrt{2}$ is square or non-square in $\mathbb{F}_q$,
$$u_0^{(q+1)/2}=-\frac{4}{\varepsilon\sqrt{2}}(2-\chi\sqrt{2})^{(q-1)/2}(2-\chi\sqrt{2})b.$$
Comparison shows the first Equation in (\ref{eqC24lug}) for $v=v_0,\,u=u_0$.
Let
\begin{equation}
\label{eq526bislug}
t_0=-2(1+\chi \sqrt{2})b,\quad s_0=\pm\frac{4}{\varepsilon\sqrt{-2}}(2+\chi \sqrt{2})b.
\end{equation}
Observe that $2+\chi\sqrt{2}$ and $2-\chi\sqrt{2}$ are simultaneously squares and non-squares in $\mathbb{F}_q$. Therefore, the above computation also shows that
the second Equation in (\ref{eqC24lug}) holds for $t=t_0,\,s=s_0$. Furthermore, $u_0^qs_0=16b^2=(t_0-v_0^q)^2$ which proves the third Equation in (\ref{eqC24lug}) for $u=u_0,v=v_0,t=t_0,s=s_0$.

Furthermore, $v_0^q=-v_0$, $v_0t_0=-4b^2$ and $v_0+t_0=-4b$ whence $(v_0+t_0)^{q+1}=-16b^2$ and $v_0t_0=(v_0t_0)^q$. Therefore, $F(v_0,t_0)=0$.  This computation shows that $P_{u_0,v_0}$,$Q_{s_0,t_0}$ is a generator $g_0\in \cG_1$. Also,
$$u_0=-\frac{v_0^2+2v_0b}{\pm\varepsilon \sqrt{-2}b}=\frac{v_0^2+2v_0b}{a^q},\quad  s_0=\frac{t_0^2+2t_0b}{\pm\varepsilon \sqrt{-2}b}=\frac{t_0^2+2t_0b}{a^q}.$$
Therefore, both points $P_{u,v}$ and $Q_{s,t}$ lie on the tangent plane $\pi_P$ to $\cU_3$ at $P$. This yields that $g_0$ passes through $P$.

Next we show how all generators $g\in \cG_1$ passing through $P$ can be obtained from $g_0$ by applying Lemma \ref{lem1Alug1}. If $g=P_{u,v}Q_{s,t}$ and $g$ passes through $P$, these three points are collinear, and hence $F(v,t)=0$ and $b(v+t)=vt$ by Lemma \ref{lemA5agos}.

Now, for $\alpha,\beta,\gamma,\delta\in \mathbb{F}_q$ with $\alpha\delta-\beta\gamma\neq 0$, let
\begin{equation}
\label{eqMlug} v=\frac{\alpha v_0+\beta}{\gamma v_0+ \delta}\,, \quad t=\frac{\alpha t_0+\beta}{\gamma t_0+ \delta}\,.
\end{equation}
Since $v_0t_0=-4b^2$ and $v_0+t_0=-4b$, Equations (\ref{eqD24lug}) and (\ref{eqH24lug}) may be re-written as
\begin{equation}
\label{eq12oct}
\begin{array}{lll}
4\omega \alpha \gamma = 2 \alpha \beta +\beta \delta,\\
\beta^2=4\omega(\alpha^2-\alpha \delta -\beta \gamma).
\end{array}
\end{equation}
We show that these equations hold if and only if  $\alpha, \beta,\gamma,\delta$ depend on a unique parameter $\xi\in \mathbb{F}_q\cup \{\infty\}$; more precisely, up to a non-zero constant in $\fq$,
\begin{equation}
\label{eq25Clug}
\alpha=\xi^2+4\omega,\,\, \beta=(4\omega+\xi^2)\xi,\,\, \gamma=\frac{\xi}{4\omega}(-\xi^2+12 \omega),\,\, \delta=4\omega-3\xi^2,
\end{equation}
and the exceptional case $\alpha=\delta=0, \beta=1,\gamma=-(4\omega)^{-1}$ corresponds to $\xi=\infty$. 

To discuss the case of $\delta\neq 0$, we may assume $\delta=1$. Then $\alpha\neq 0$, as $\alpha\delta-\beta\gamma\neq 0$. Now, the first Equation in (\ref{eq12oct}) gives 
$$\gamma=\frac{(2\alpha+1)\beta}{4\omega\alpha}.$$ This together with the second Equation in (\ref{eq12oct}) yield $\alpha\beta^2=4\omega(\alpha^3-\alpha^2)-(2\alpha+1)\beta^2$
whence $4\omega \alpha^3-3\alpha\beta^2-4\omega \alpha^2-\beta^2=0$. Let $\xi=\beta\alpha^{-1}$. Then $\alpha^2(4\omega\alpha-3\alpha\xi^2-4\omega -\xi^2)=0.$ Therefore, 
$$\alpha=\frac{\xi^2+4\omega}{4\omega-\xi^2},$$ whence the assertion follows for $\delta\neq 1$. If $\delta=0$ then we may assume $\beta=1$. Then $\alpha=0$. Otherwise, the first Equation in (\ref{eq12oct}) yields $\gamma=(2\omega)^{-1}$. Now the second Equation in (\ref{eq12oct}) reads $4\omega\alpha^2=-1$ which is impossible as $-1$ is square while $\omega$ is non-square in $\mathbb{F}_q$. Finally, if $\delta=\alpha=0$ and $\beta=1$ then $\gamma=-(4\omega)^{-1}$. 

Therefore,
\begin{equation}
\label{eqL24lug}
v=v_\xi=\frac{(\xi^2+4\omega)v_0+(4\omega+\xi^2)\xi}{\frac{\xi}{4\omega}(-\xi^2+12 \omega)v_0+4\omega-3\xi^2},\quad v_\infty=\frac{1}{-\frac{1}{4\omega} v_0}=-2(1+\chi\sqrt{2})b.
\end{equation}
Furthermore, the determinant of the associated linear fractional map equals
\begin{equation}
\label{eqK24lug} \det(\xi)=\frac{(\xi^2+4\omega)(\xi^4-24\omega \xi^2+16\omega^2)}{4\omega},\quad \det(\infty)=(4\omega)^{-1}.
\end{equation}
 Equation (\ref{eqL24lug}) remains valid if $v,v_0$ are replaced by $t,t_0$:
\begin{equation}
\label{eqS24lug}
t_\xi=\frac{(\xi^2+4\omega)t_0+(4\omega+\xi^2)\xi}{\frac{\xi}{4\omega}(-\xi^2+12 \omega)t_0+4\omega-3\xi^2}\,\quad t_\infty=\frac{1}{-\frac{1}{4\omega} t_0}=-2(1-\chi\sqrt{2})b.
\end{equation}

We show that Equations (\ref{eqF24lug}) impose a restriction on $\xi$ in (\ref{eqL24lug}) and (\ref{eqS24lug}):
\begin{lemma}
\label{lemB26lug}
$\xi^2+4\omega$ is a square or a non-square in $\mathbb{F}_q$ according as $\varepsilon=1$ or $\varepsilon=-1$.
\end{lemma}
\begin{proof} To use Lemma \ref{lemA26lug} we need to express $v_2^2+\frac{1}{\omega}v_1^2+2(1\pm\varepsilon \sqrt{2})v_2$ in terms of $\xi$. This requires a certain amount of tedious but straightforward computations.

From (\ref{eqL24lug}),
\begin{equation}
\label{eqT24lug}
v_1=\frac{-2\sqrt{2}}{\kappa}(\xi^2+4\omega)\xi,\quad
v_2=\frac{\sqrt{2}-1}{2\omega\kappa}(\xi^2+4\omega)(\xi^2-4\omega(\sqrt{2}+1)^2).
\end{equation}
with
$$\kappa=\frac{-(\sqrt{2}-1)^2}{4\omega}\,\xi^4+10\xi^2-\frac{4\omega}{(\sqrt{2}-1)^2}.$$
Therefore,
$$
v_2^2+\frac{1}{\omega}v_1^2=\frac{(\sqrt{2}-1)^2}{4\omega^2 \kappa^2}(\xi^2+4\omega)^3(\xi^2+4\omega(\sqrt{2}+1)^4)=
$$

$$
=\frac{(\sqrt{2}-1)^2}{4\omega^2\kappa^2}(\xi^2+4\omega)(\xi^6+4\omega(2+(\sqrt{2}+1)^4)\xi^4+16\omega^2(1+2(\sqrt{2}+1)^4)\xi^2+64\omega^3(\sqrt{2}+1)^4)
$$
and
$$
\kappa v_2=\frac{\sqrt{2}-1}{2\omega \kappa^2}(\xi^2+4\omega)\left(-\frac{(\sqrt{2}-1)^2}{4\omega}\xi^6+11 \xi^4-44(\sqrt{2}+1)^2\omega \xi^2+16\omega^2(\sqrt{2}+1)^4\right).
$$
For $\pm\varepsilon=1$, this gives
$$v_2^2+{\frac{1}{\omega}}v_1^2+2(1+\sqrt{2})v_2={\frac{\sqrt{2}-1}{\omega \kappa^2}}(\xi^2+4\omega)\left({\frac{3\sqrt{2}-4}{2\omega}}\xi^6-4(\sqrt{2}-4)\xi^4+\omega(88\sqrt{2}+96)\xi^2\right)$$
whence
\begin{equation}
\label{eq31lug}
v_2^2+{\frac{1}{\omega}}v_1^2+2(1+\sqrt{2})v_2={\frac{10-7\sqrt{2}}{2\omega^2\kappa^2}}(\xi^2+(20+16\sqrt{2})\omega)^2\xi^2 (\xi^2+4\omega).
\end{equation}

For $\pm\varepsilon=-1$, this gives
$$v_2^2+\frac{1}{\omega}v_1^2+2(1-\sqrt{2})v_2=\frac{\sqrt{2}-1}{\omega\kappa^2}(\xi^2+4\omega)((16+18\sqrt{2})\xi^4-\omega(256+176\sqrt{2})\xi^2+\omega^2(768+544\sqrt{2}));$$
whence
\begin{equation}
\label{eqU25lug}
v_2^2+{\frac{1}{\omega}}v_1^2+2(1-\sqrt{2})v_2={\frac{20-2\sqrt{2}}{\omega\kappa^2}}(\xi^2-\frac{20+16\sqrt{2}}{7}\omega)^2 (\xi^2+4\omega).
\end{equation}
Now Lemma \ref{lemB26lug} follows from Equations (\ref{eq31lug}) and (\ref{eqU25lug}) together with Lemma \ref{lemA26lug}.
\end{proof}

To state a corollary of Lemmas \ref{lem27lug}, \ref{lemB26lug} and \ref{lemA26lug}, the partition of $\mathbb{F}_q\cup\{\infty\}$ into two subsets $\Sigma_1\cup\{\infty\}$ and $\Sigma_2$ is useful where $x\in\Sigma_1$ or $x\in\Sigma_2$ according as $x^2+4\omega$ is square or non-square in $\mathbb{F}_q$.
\begin{proposition}
\label{propA2agos2017}
Let $P=(\varepsilon\sqrt{-2}\,b,b,0)\in \cU_3$ with $b^q+b=0$ and $b\neq 0$. Then the generators in $\cG_1$ through $P$ which meet $\cX^+$ are as many as
$n_P=\ha(q+1)$. They are exactly the lines $g_\xi$ joining $P$ to $P_{u,v}=(u,v,v^2)$ with $u,v$ as in (\ref{eqZ5agos}) and (\ref{eqL24lug}) where $\xi$ ranges over $\Sigma_1\cup\{\infty\}$ or $\Sigma_2$ according as $\varepsilon=1$ or $\varepsilon=-1$.
\end{proposition}

\subsection{Case of $\cG_2$}  The tangent plane $\pi_P$ to $\cU_3$ at $P$ meets $\pi$ in the line $r$ with equation $2b^qY+Z=0$. Since $\cC$ has equation $Z=Y^2$ in $\pi$, the common points of $r$ and $\cC$ are the origin $O=(0,0,0)$ and the point $B=(0,2b,4b^2)$. Here  $B\not\in \Omega$ by $b\not\in\mathbb{F}_q$. This gives the following result.
\begin{proposition}
\label{prop2agos2017} Let $P=(\varepsilon\sqrt{-2}\,b,b,0)\in \cU_3$ with $b^q+b=0$ and $b\neq 0$. Then there is a unique generator through the point $P$ which meets $\Omega$, namely the line $\ell$ through $P$ and the origin $O=(0,0,0)$.
\end{proposition}
\subsection{Choice of $\cM_1$ and $\cM_2$} So far, the symbols $\cM_1$ and $\cM_1'$, as well as $\cM_2$ and $\cM_2'$, have been used to name generically the two $\Phi$-orbits on $\cG_1$ and on $\cG_2$, respectively. Furthermore, the symbol $\cM$ has denoted the union of any of the two $\Phi$-orbits on $\cG_1$ with any of the two $\Phi$-orbits on $\cG_2$. Therefore, our results in the previous sections apply if we fix anyway a $\Phi$-orbit on $\cG_1$ and name it $\cM_1$, and then take any of the two $\Phi$-orbits on $\cG_2$ and name it $\cM_2$.

To explain how to choose $\cM_1$ and $\cM_2$ in the best possible way, we need some more facts on the generator $g_0$.
\begin{remark}
\label{rem7agos} {\em{By definition, we have two different generators $g_0$, one for $\varepsilon=1$ another for $\varepsilon=-1$. In fact,
\begin{equation}
\label{eq110agos}
\begin{array}{lllll}
P_{u_0,v_0}=
\begin{cases}
{\mbox{$(\pm\frac{-4}{\sqrt{-2}}(2-\chi \sqrt{2})b,-2(1-\chi\sqrt{2})b,4(1-\chi\sqrt{2})^2b^2)$, for $\varepsilon=1$}},\\
{\mbox{$(\pm\frac{4}{\sqrt{-2}}(2-\chi \sqrt{2})b,-2(1-\chi\sqrt{2})b,4(1-\chi\sqrt{2})^2b^2)$, for $\varepsilon=-1$}};
\end{cases}
\\
Q_{s_0,t_0}=
\begin{cases}
{\mbox{$(\pm\frac{4}{\sqrt{-2}}(2+\chi \sqrt{2})b,-2(1+\chi\sqrt{2})b,4(1+\chi\sqrt{2})^2b^2)$, for $\varepsilon=1$}},\\
{\mbox{$(\pm\frac{-4}{\sqrt{-2}}(2+\chi \sqrt{2})b,-2(1+\chi\sqrt{2})b,4(1+\chi\sqrt{2})^2b^2)$, for $\varepsilon=-1$.}}
\end{cases}
\end{array}
\end{equation}
This motivates to adopt two different symbols for $g_0$, say $g_0^+$ and $g_0^-$, when comparison gives useful information. Observe that $g_0^+$ passes through $P(\sqrt{-2}b,b,0)$ while $g_0^-$ does through $P(-\sqrt{-2}b,b,0)$.
}}
\end{remark}
\begin{lemma}
\label{lem7agos} The generators $g_0^+$ and $g_0^-$ are in different orbits of $\Phi$.
\end{lemma}
\begin{proof} From (\ref{eq110agos}), the linear collineation associated to the matrix $\mathbf{W}$ interchanges $g_0^+$ and $g_0^-$. Therefore, the assertion follows from Lemma \ref{lem19lug2017}.
\end{proof}
Let $r$ (resp. $r'$) denote  the number of generators in $\cM_1$ (resp. $\cM_1'$) through the point $P=(\sqrt{-2}b,b,0)$ that meet $\Delta^+$. Here $r'=\ha(q+1)-r$ by Proposition \ref{propA2agos2017}. Furthermore,
Let $\ell^+$ (resp $\ell^-$) denote the generator joining the origin $O$ to $P=(\sqrt{-2}b,b,0)$ (resp. $P=(-\sqrt{-2}b,b,0)$).
\begin{lemma}
\label{lem12ago2017} The generators $\ell_0^+$ and $\ell_0^-$ are in different orbits of $\Phi$.
\end{lemma}
\begin{proof} Look at the set of generators $g_k\in \cG_2$ through $Z_\infty$. From claims (B) and (D) in the proof of Lemma \ref{lemA5lug2017}, the linear collineation $\mathfrak{g}\in \mathfrak{G}\setminus\mathfrak{H}$ associated to $\mathbf{N}_{-1}\mathbf{R}_\mu$ as well as $\mathbf{w}$ take $g_k$ to $g_{-k}$. Since $\mathbf{w}$ commutes with $\mathfrak{G}$ and $[\mathfrak{G}:\mathfrak{H}]=2$, this and Lemma \ref{lemA5lug2017} yield that $\mathbf{w}$ interchanges the two $\mathfrak{H}$-orbits on $\cG_2$. On the other hand, $\mathbf{w}$ fixes the origin $O$, while it interchanges $P=(\sqrt{-2\,}b,b,0)$ with $P=(-\sqrt{-2}\,b,b,0)$. Hence $\mathbf{w}$ interchanges $\ell_0^+$ with $\ell_0^-$, as well. This proves the lemma.
\end{proof}
Now we are in a position to make our choices of $\cM_1$ and $\cM_2$.
\begin{itemize}
\item[(vii)] {\em{ $\cM_1$ is the $\Phi$-orbit containing $g_0^+$.}}
\end{itemize}
\begin{itemize}
\item[(viii)] {\em{$\cM_2$ is the $\Phi$-orbit containing $\ell^+$ for $r<r'$, and $\ell^-$ for  $r>r'$.}}
\end{itemize}
 We show that  $r'$ may be obtained counting the squares in the value set of the polynomial $f(\xi)$  defined in Theorem \ref{teo5ago2017}.
\begin{proposition}
\label{prop11agos}  The number of $\xi\in \mathbb{F}_q$ for which
$\xi^4-24\omega \xi^2+16\omega^2$ is square in $\mathbb{F}_q$ equals $2r'-1$.
\end{proposition}
\begin{proof} From Proposition \ref{propA2agos2017}, $r$ also counts $\xi\in \mathbb{F}_q\cup \{\infty\}$ for which $\det(\xi)$ is square. From Lemma \ref{lem7agos},
$r$ counts, as well, the generators in $\cM_1'=\cG_1\setminus \cM_1$ through $P=(-\sqrt{-2}b,b,0)$ that meet $\Delta^+$. If $\cM_1$ and $P=(\sqrt{-2}b,b,0)$ are replaced by $\cM_1'$ and  $P=(-\sqrt{-2}b,b,0)$, this remains
true for $r'$. From Proposition \ref{propA2agos2017},  $r+r'=\ha(q+1)$. On the other hand, (\ref{eqK24lug}) shows in $\mathbb{F}_q$ that  $f(\xi)$ is square if and only if either $\det(\xi)$ is square and $\xi^2+4\omega$ is non-square, or $\det(\xi)$ is non-square and $\xi^2+4\omega$ is square. Since $\det(\infty)$ is non-square, this yields that $f(\xi)$ is non-square $2r$ times while $f(\xi)$ is square $2r'-1$ times.
\end{proof}
Finally, we show that Theorem \ref{teo5ago2017} is indeed a corollary of the propositions proven in this Section. From Propositions \ref{propA2agos2017} and \ref{prop2agos2017}, $r+r'=\ha (q+3).$  Then Condition (B) holds if and only if exactly half of them
is in $\cM_1\cup \cM_2$ (and hence the other half in $\cM_1'\cup \cM_2'$). From Proposition \ref{prop2agos2017}, we have only two possibilities, namely $r=\qa(q-1), r'=\qa(q+3)$ and $r=\qa(q+3),r'=\qa(q-1)$. Thus Proposition \ref{teo5ago2017} follows from Proposition \ref{prop11agos}.
\begin{lemma}
\label{lem5ago} If Condition $\rm{(B)}$ is satisfied by a point $P$ then it is satisfied by all points in the $\mathfrak{G}\times \langle \mathfrak{w}\rangle$ orbit containing $P$.
\end{lemma}
\begin{proof} The subgroup $\mathfrak{H}$ of is a normal subgroup of $\mathfrak{G}$. This together with (ii) of Lemma \ref{lem4lug2017} shows that the assertion holds for the images of $P$ by elements in $\mathfrak{G}$. Actually, since $\mathfrak{H}$ is also normalized by $\mathfrak{w}$, this remains valid for the image $\mathfrak{w}(P)$ of $P$ by Remark \ref{rem26lug} where it is observed that
 $\mathfrak{w}$ interchanges $\cX^+$ and $\cX^-$ and hence preserves the generator set $\cG_1$.
\end{proof}

By Proposition \ref{teo5ago2017} we are led to compute the number $N_q(\cC_4)$ of all points of the plane curve $\cC_4$
\begin{equation}
\label{eq13agos2017} Y^2=X^4-24\omega X^2+16\omega^2
\end{equation}
which lie in the affine plane $AG(2,\mathbb{F}_q)$ with affine coordinates $(X,Y)$. A direct computation shows the following basic properties of $\cC_4$. It is an irreducible quartic with only one singular point which is $Y_\infty$, a double-point, the center of two branches (places) of $\cC_4$ both defined over $\mathbb{F}_q$. Let $N_q(\cC_4)$ denote the number of its $\mathbb{F}_q$-rational points. Then the following result holds.
\begin{lemma}
\label{lem14ago2017} The number of $\xi\in \mathbb{F}_q$ for which $\xi^4-24\omega \xi^2+16\omega^2$ is square in $\mathbb{F}_q$ is equal to $\ha (N_q(\cC_4)-2).$
\end{lemma}
 For the purpose of dealing with a simpler equation the following lemma is useful.
\begin{lemma}
\label{lem14Aago2017}
The plane quartic $\cC_4$ is birationally equivalent over $\mathbb{F}_q$ to the elliptic curve $\cC_3$ with Weierstrass equation
\begin{equation}
\label{eq13Aagos2017} \frac{1}{2\omega}Y^2=X^3-X.
\end{equation}
\end{lemma}
\begin{proof} The quadratic transformation  $(X,Y)\mapsto (X,Y+X^2)$  maps $\cC_4$ to the plane cubic curve $\cC_3$ with affine equation $2X^2Y+24\omega X^2+Y^2=16\omega^2.$  Now
the invertible rational transformation $$(X,Y)\mapsto \left(\frac{X}{2(Y+12\omega)},Y\right)$$ maps $\cC_3$ to the plane cubic $\cD_3$ with affine equation $X^2=(Y^2-16\omega^2)(Y+12\omega)$. Then, the linear transformation
$(X,Y)\mapsto (X,Y-4\omega)$ takes $\cD_3$ to the elliptic curve $\cF_3$ with affine equation $X^2=Y^3-64\omega^2 Y$. Finally, the linear transformation $(X,Y)\mapsto (16\omega Y,8\omega X)$ takes $\cF_3$ to the elliptic curve $\cC_3$ with Weierstrass equation (\ref{eq13Aagos2017}). All the above transformations are defined over $\mathbb{F}_q$.
\end{proof}
Let $N_q(\cC_3)$ denote the number of $\mathbb{F}_q$-rational points of $\cC_3$. Since $\cC_3$ has exactly one point at infinity and that point is an $\mathbb{F}_q$-rational point, Lemmas \ref{lem14ago2017} and \ref{lem14Aago2017} show that $\ha (N_q(\cC_4)-2)=\ha(N_q(\cC_3)-1)$. In terms of $N_q(\cC_3)$, Proposition \ref{teo5ago2017} states that  Condition (B) for Case $\rm{(III)}$ is satisfied if and only if $N_q(\cC_3)$ equals either $q-1$ or $q+3$.

Let $\cE_3$ be the twist of $\cC_3$, that is, the elliptic curve with Weierstrass equation
\begin{equation}
\label{eq20Aagos2017} Y^2=X^3-X.
\end{equation}
Then $N_q(\cC_3)+N_q(\cE_3)=2q+2$. Thus, Theorem \ref{teo5ago2017B} is a Corollary of Proposition \ref{teo5ago2017}.
\section{Proof of Theorem \ref{teomain}}
\label{ser}
The case $q=p$ is investigated using some known results from \cite[Section 2.2.2, pg.10]{serre}. From (\ref{eqE24lug}),  $p\equiv 1 \pmod 8$. Write $p=\pi \bar{\pi}$ with $\pi\in \mathbb{Z}[i]$ where $\mathbb{Z}[i]$ is the ring of Gauss integers. Also, write $\pi=\alpha_1+i\alpha_2$. Here $\pi$ can be chosen in a unique way (up to conjugation) such that $\pi \equiv 1 \pmod{(-2+2i}$. Then $N_p(\cE_3)=q+1-2\alpha_1$. This shows that  Condition (B) for Case $\rm{(III)}$ is satisfied if and only if $p=1+4a^2$ with $a\in \mathbb{Z}$ (and $N_p(\cE_3)=p-1$). Therefore, Theorem \ref{teomain} is a corollary of Theorems \ref{teo6lug2017}, \ref{em24lug2017} and  \ref{teo5ago2017}.
It is natural to ask whether there is an infinite sequence of such primes. Obviously, an affirmative answer would imply the solution of the famous conjecture dating back to Landau about the existence of infinite primes of the form $1+n^2$ with $n\in \mathbb{Z}$. For a detailed discussion on recent progress on Landau's conjecture, see \cite{Pi}.

The case of $q=p^h$ with $h>1$ and $p\equiv 1 \pmod 4$ remains open. Our computer aided search did not provide hemisystems. It may be observed that no hemisystem arises for even $h$ if the above quoted result for $p$ holds true for $q$, that is, $N_q(\cE_3)=q+1-2a_1$ if and only if $q=a_1^2+a_2^2$ with $a_1,a_2\in \mathbb{Z}$. This follows from the classical result of Lebesque which
states that if $n^2+1=p^r$ then $r=1$; see, for instance, \cite[pg.276]{de}.

 Finally, we point out that no hemisystem arises for $q=p^{2h}$ with $p\equiv 3 \pmod 4$. In this case, $N_p(\cE_3)=p+1$, see \cite[Section 2.2.2, pg.10]{serre}. Therefore, the Zeta function $Z(E_3/\mathbb{F}_p,T)$ is equal to $(1+pT^2)/((1-T)(1-pT))$, and hence
$N_q(\cE_3)=p^{2h}+1+2p^h=q+1+2p^h$; see \cite[12.2.105 Remark]{handbook}. Comparison with Theorem \ref{teo5ago2017} shows the assertion.


\begin{thebibliography}{99}
\bibitem{bb} J. Bamberg, M. Lee, K. Momihara and Q. Xiang, A new infinite family of hemisystems of the Hermitian surface, \emph{Combinatorica}  doi:10.1007/s00493-016-3525-4.
\bibitem{bb1} J. Bamberg, M. Giudici and Royle, Every flock generalized quadrangle has a hemisystem, \emph{Bull. London Math. Soc.} {\bf{42}} (2010), 795-810.
\bibitem{bgr} J. Bamberg, M. Giudici and Royle, Hemisystems of small flock generalized quadrangles, \emph{Des.Codes Cryptogr.} {\bf{67}} (2013), 137-157.
\bibitem{cp} A. Cossidente and T. Penttila,  Hemisystems on the Hermitian surface, \emph{J. London Math. Soc.} (2) {\bf{72}} (2005), 731-741.
\bibitem{co} A. Cossidente,  Combinatorial structures in finite classical polar spaces, in \emph{Surveys in Combinatorics 2017}, pg. 204-237, LMS Lecture Note Series {\bf{440}}, 2017.
\bibitem{cpa} A. Cossidente and F. Pavese, Intriguing sets of quadrics in $PG(5, q)$, \emph{Adv. Geom} {\bf{17}} (2017), 339–-345.
\bibitem{de} P. Dembowski, \emph{Finite geometries}, Springer, Berlin, 1968; Reprint, 1997.
\bibitem{giuzzikorchmaros} L. Giuzzi and G. Korchm\'aros, Ovoids of the Hermitian surface in odd characteristic, \emph{Adv. Geom.} {\bf{3}} (2003), 251-261.
\bibitem{hering} Ch. Hering, A theorem on group spaces, \emph{Hokkaido Mathematical Journa} {\bf{8}} (1979), 115-120.
\bibitem{HKT}J.W.P.~ Hirschfeld -- G.~Korchm\'aros and F.~Torres, \emph{Algebraic curves over a finite field}, Princeton University Press,
Princeton, N.J, 2008.
\bibitem{KT}  G. Korchm\'aros and  F. Torres: Embedding of a maximal
curve in a Hermitian variety, \emph{Compositio Mathematica}, {\bf 128}
(2001), 95-113.
\bibitem{handbook} (eds) G.L. Mullen, D. Panario, \emph{Handbook of Finite Fields}, CRC Press, Taylor \& Francis, Boca Raton, 2013.
\bibitem{Pi} J. Pintz, Landau's problems on primes, \emph{Journal de Théorie des Nombres de Bordeaux} {\bf{21}} (2009), 357-404.
\bibitem{segre} B. Segre, Forme e geometrie hermitiane, con particolare riguardo al caso finito.
\emph{Ann. Mat.  Pura Appl.} (4) {\bf{70}} (1965), 1-201.
\bibitem{serre} J.P. Serre, \emph{Lectures on $N_X(p)$}, CRC Press, Taylor \& Francis, Boca Raton, 201.
\bibitem{thas} J.A. Thas, Projective geometry over a finite field, in \emph{Handbook of incidence geometry}, 295-347.
North-Holland, Amsterdam, 1995.
\end{thebibliography}
\end{document}